\documentclass[leqno]{article}

\usepackage[frenchb,english]{babel}
\usepackage[utf8]{inputenc}
\usepackage{amsmath}
\usepackage{amssymb}
\usepackage{amsfonts}
\usepackage{enumerate}
\usepackage{vmargin}
\usepackage[all]{xy}
\usepackage{mathrsfs}
\usepackage{mathtools}
\usepackage{lmodern}
\usepackage{slashed}
\usepackage[colorlinks=true,linkcolor=blue,pagebackref=true]{hyperref}%
\setmarginsrb{3cm}{3cm}{3.5cm}{3cm}{0cm}{0cm}{1.5cm}{3cm}
\usepackage{comment}
\usepackage{tikz}

\usepackage{accents}

\newlength{\dhatheight}

\newcommand{\D}{\mathrm{D}}

\newcommand{\I}{\mathrm{I}} 

\newcommand{\M}{\mathrm{M}}

\let\cal\relax
\newcommand{\cal}{\mathcal}

\newcommand{\R}{\ensuremath{\mathbb{R}}}

\newcommand{\Id}{\mathrm{Id}}

\newcommand{\la}{\langle}
\newcommand{\ra}{\rangle}

\renewcommand{\leq}{\ensuremath{\leqslant}}
\renewcommand{\geq}{\ensuremath{\geqslant}}
\newcommand{\qed}{\hfill \vrule height6pt  width6pt depth0pt}
\newcommand{\bnorm}[1]{ \big\| #1  \big\|}

\newcommand{\norm}[1]{\left\Vert#1\right\Vert}

\newcommand{\co}{\colon}

\newcommand{\ot}{\otimes}
\newcommand{\ovl}{\overline}

\newcommand{\cb}{\mathrm{cb}}

\let\i\relax 
\newcommand{\i}{\mathrm{i}}

\newcommand{\ov}{\overset}

\newcommand{\epsi}{\varepsilon}
\newcommand{\e}{\mathrm{e}} 

\DeclareMathOperator{\tr}{Tr} 
\DeclareMathOperator{\Tr}{Tr} 
\let\Re\relax 
\DeclareMathOperator{\Re}{Re} 

\newtheorem{thm}{Theorem}[section]

\newtheorem{prop}[thm]{Proposition}

\newtheorem{lemma}[thm]{Lemma}

\newtheorem{remark}[thm]{Remark}

\newenvironment{proof}[1][]{\noindent {\it Proof #1} : }{\hbox{~}\qed
\smallskip
}

\usepackage{tocloft}
\numberwithin{equation}{section}
\usepackage[nottoc,notlot,notlof]{tocbibind}

\let\OLDthebibliography\thebibliography
\renewcommand\thebibliography[1]{
  \OLDthebibliography{#1}
  \setlength{\parskip}{0pt}
  \setlength{\itemsep}{0pt plus 0.3ex}
}

\newcommand\reallywidehat[1]{\arraycolsep=0pt\relax%
\begin{array}{c}
\stretchto{
  \scaleto{
    \scalerel*[\widthof{\ensuremath{#1}}]{\kern-.5pt\bigwedge\kern-.5pt}
    {\rule[-\textheight/2]{1ex}{\textheight}} 
  }{\textheight} %
}{0.5ex}\\           
#1\\                 
\rule{-1ex}{0ex}
\end{array}
}

\begin{document}
\selectlanguage{english}

\title{\bfseries{Optimal square size for separating the operator and completely bounded norms of Schur multipliers on $S^4$}}
\date{}
\author{\bfseries{C\'edric Arhancet}}%
\maketitle
%


\begin{abstract}
We determine the smallest square size at which the operator norm and the completely bounded norm of a Schur multiplier on $S^4$ can differ. More precisely, for
$A=\begin{bmatrix}1&2&0\\1&-2&0\\2&2\mathrm{i}&0\end{bmatrix}$
we prove
\[
\|M_A\|_{S^4_3\to S^4_3}^2\leq\frac{1119}{250}
<\left(\frac{18983532}{947485}\right)^{1/2}
\leq \|\Id_{S^4_2}\ot M_A\|_{S^4_6\to S^4_6}^2.
\]
We also prove that no $2\times2$ Schur multiplier can have unequal operator and completely bounded norms on $S^4$, so square size three is optimal. In addition, we determine the exact ordinary norm of the previous multiplier analytically and give an independent computer-assisted certification of a rational upper bound.
\end{abstract}


\makeatletter
 \renewcommand{\@makefntext}[1]{#1}
 \makeatother

{
  \hypersetup{linkcolor=blue}
 \tableofcontents
}

\section{Introduction}
\label{sec:Introduction}

For $1\leq p<\infty$, let $S^p_n$ denote $\M_n$ equipped with the Schatten $p$-norm $\|x\|_{S^p_n}=(\tr|x|^p)^{\frac{1}{p}}$. A scalar matrix $A=[a_{ij}] \in \M_n$ defines the Schur multiplier
\[
M_A \co S^p_n \to S^p_n,\qquad M_A(x)=[a_{ij}x_{ij}]_{i,j=1}^n.
\]
Its completely bounded norm is
\[
\norm{M_A}_{\cb,S^p_n \to S^p_n}
\ov{\mathrm{def}}{=} \sup_{m \geq 1} \norm{\Id_{S^p_m} \ot M_A}_{S^p_{mn}\to S^p_{mn}}.
\]
The equality between the ordinary and completely bounded norms of Schur multipliers is automatic for $p=1,2,\infty$. For $1<p<\infty$, $p\neq2$, Pisier asked whether there exists a bounded Schur multiplier on $S^p$ which is not completely bounded \cite[Conjecture 8.1.12]{Pis98} (see also \cite[Problem 9.8 p.~1503]{PiX03}). A preliminary finite-dimensional problem is whether one can have
\[
\norm{M_A}_{S^p_n\to S^p_n}
<
\norm{M_A}_{\cb,S^p_n \to S^p_n}
\]
for some integer $n$ and some symbol $A \in \M_n$. This question was emphasized by Lafforgue and de la Salle \cite[after Conjecture 1.8]{LafforgueDeLaSalle2011} and by Caspers and Wildschut \cite[Section 5]{CaspersWildschut2019}. We refer to \cite{AlP20} for the case $0< p <1$ with a different behavior and to \cite{Arh12} for the class of Fourier multipliers.

Every linear map between finite-dimensional operator spaces is completely bounded, so a finite-dimensional strict inequality does not resolve Pisier's conjecture. It isolates the genuinely isometric distinction between boundedness and complete boundedness that must precede such an example.

In parallel with, and independently of, the development of the present work, Huang, Sukochev and Tomskova very recently obtained the first example at $p=4$ in \cite{HuangSukochevTomskova2026}, using a different approach. Their symbol has a nonzero $3\times4$ corner. The purpose of the present paper is to determine the smallest square size at which this phenomenon can occur. We obtain a simpler $3\times3$ symbol and prove that size two is impossible. Our main result is the following theorem.

\begin{thm}
\label{thm-main}
Let $
A\ov{\mathrm{def}}{=} \begin{bmatrix}1&2&0\\1&-2&0\\2&2\mathrm{i}&0
\end{bmatrix}$. Then
\begin{equation}
\label{eq-main-gap}
\|M_A\|_{S^4_3\to S^4_3}
\leq\sqrt{\frac{1119}{250}}
<\left(\frac{18983532}{947485}\right)^{\frac{1}{4}}
\leq \|\Id_{S^4_2}\ot M_A\|_{S^4_6 \to S^4_6}.
\end{equation}
In particular, we have $\norm{M_A}_{S^4_3 \to S^4_3} < \norm{M_A}_{\cb,S^4_3\to S^4_3}$.
\end{thm}

Numerically, the two certified bounds in \eqref{eq-main-gap} are
\[
2.1156559266\ldots
<2.1156857887\ldots.
\]

The scalar upper bound admits two independent proofs. We first give a fully analytic argument which, in fact, determines the exact ordinary norm of the multiplier. We then retain a computer-assisted proof of the slightly weaker rational estimate used in Theorem~\ref{thm-main}, both as an independent verification and because of its rather different convex-geometric nature. In the latter argument, floating-point computations are used solely to propose rational certificates, while every inequality used to certify a box is checked in exact rational arithmetic. The lower bound for the second amplification and the optimality of the square size are entirely analytic.


\paragraph{Approach of the paper.}
Our approach is based on separating the structural part of the problem from the particular choice of the symbol. In contrast with the scalar upper-bound argument of Huang, Sukochev and Tomskova \cite{HuangSukochevTomskova2026}, which exploits the special phase pattern of their symbol through estimates between row correlations and a subsequent low-dimensional eigenvalue calculation, we first derive a variational principle valid for an arbitrary two-column symbol. More precisely, the $S^4$-norm problem is transformed into an optimization problem over positive $2\times2$ matrices: convex geometry of finite positive operator-valued measures removes the rank-one constraints, Hilbert--Schmidt duality and conic duality then identify the norm with a minimization problem. For the particular symbol considered here, symmetry reduces this semidefinite problem to a one-parameter optimization, from which the exact scalar norm can be determined analytically. A similar structural principle is used to rule out the two-by-two case: a homogeneous scalar quartic inequality is lifted to arbitrary matrix amplifications by means of convex weights and H\"older's inequalities. Thus, rather than estimating individual matrices directly, the proofs reduce the distinction between ordinary and completely bounded norms to finite-dimensional convex geometry and elementary optimization.

\paragraph{Structure of the paper}
The paper is organized as follows. Section~\ref{sec-preliminaries} gives the elementary finite-corner reductions. Section~\ref{sec-M2} proves the optimality theorem.  Section~\ref{sec-variational} derives a two-dimensional variational formula for the scalar norm. Section~\ref{sec-analytic-scalar-norm} uses this formula to give a fully analytic computation of the exact scalar norm. Section~\ref{sec-certificate} provides an independent computer-assisted proof of the rational upper estimate appearing in Theorem~\ref{thm-main}, together with the explicit lower witness for the second amplification. Finally, Section~\ref{Reproducibility} explains the computational procedure and provides the information needed to reproduce the certified estimate.

\section{Finite-corner reductions}
\label{sec-preliminaries}

We use the unnormalized trace throughout. For a rectangular matrix $x\in \M_{m,n}$, the notation $\|x\|_p$ refers to the $\ell^p$-norm of its singular values. We write $S^p_{m,n}$ for $\M_{m,n}$ with this norm.

If a square symbol has zero rows or columns, its multiplier norm is determined by the corresponding rectangular corner. We shall use the following immediate observation.

\begin{lemma}
\label{lem-corner}
Let $B\in \M_{m,n}$ and let $A\in \M_N$ be obtained by placing $B$ in a corner and setting all remaining entries equal to zero, where $N\geq\max(m,n)$. Then
\[
\|M_A\|_{S^p_N\to S^p_N}
=\|M_B\|_{S^p_{m,n}\to S^p_{m,n}}.
\]
The analogous equality holds at every matrix level.
\end{lemma}

For the symbol in Theorem~\ref{thm-main}, it is therefore enough to study
\begin{equation}
\label{eq-B}
B=\begin{bmatrix}
1&2\\1&-2\\2&2\mathrm{i}\end{bmatrix}.
\end{equation}
At matrix level two, the symbol is $J_2\otimes A$, where $J_2$ is the $2\times2$ all-ones matrix. After applying the tensor flip, this symbol becomes $A\otimes J_2$. Rectangular compressions of an amplification give lower bounds for its norm.

\section{No two-by-two example exists}
\label{sec-M2}

We begin with the optimality part of the problem. We show that the phenomenon exhibited in Theorem~\ref{thm-main} cannot occur for a two-by-two symbol: every Schur multiplier on $S^4_2$ has equal ordinary and completely bounded norms. The main ingredient is a lifting principle which converts a scalar homogeneous quartic inequality into the corresponding inequality for matrices of arbitrary size. This allows us to control all matrix amplifications using only the scalar $S^4_2$ estimate.

\begin{lemma}
\label{lem-lifting}
Suppose that $b_{ij},u_i,v_j>0$ and $K\geq0$ satisfy
\begin{equation}
\label{eq-scalar-polynomial}
\sum_{i,j=1}^2b_{ij}x_{ij}^4+2\sum_{i=1}^2u_ix_{i1}^2x_{i2}^2+2\sum_{j=1}^2v_jx_{1j}^2x_{2j}^2\geq Kx_{11}x_{12}x_{21}x_{22}
\end{equation}
for all $x_{ij}\geq0$. Let $X_{ij} \in \M_m$ and put
\begin{equation}
\label{def-qij-ri}
q_{ij}=\|X_{ij}\|_4^4,
\quad r_i \ov{\mathrm{def}}{=}\|X_{i1}^*X_{i2}\|_2^2,
\quad c_j \ov{\mathrm{def}}{=}\|X_{1j}X_{2j}^*\|_2^2,
\quad \tau\ov{\mathrm{def}}{=}\Tr(X_{12}^*X_{11}X_{21}^*X_{22}).
\end{equation}
Then
\begin{equation}
\label{eq-operator-polynomial}
\sum_{i,j=1}^2b_{ij}q_{ij}+2\sum_{i=1}^2u_ir_i+2\sum_{j=1}^2v_jc_j
\ov{\eqref{eq-operator-polynomial}}{\geq} K|\tau|.
\end{equation}
\end{lemma}

\begin{proof}
Set 
\begin{equation}
\label{def-alpha-23}
(x_1,x_2,x_3,x_4)
\ov{\mathrm{def}}{=}
(x_{11},x_{12},x_{21},x_{22})
\quad \text{and} \quad \alpha \ov{\mathrm{def}}{=}(1,1,1,1).
\end{equation}
Recall that the exponent vector of a monomial $x_1^{a_1}x_2^{a_2}x_3^{a_3}x_4^{a_4}$ is the vector $(a_1,a_2,a_3,a_4)$. Thus, if $\nu_k$ denotes the exponent vector of the $k$th monomial in the left-hand side of \eqref{eq-scalar-polynomial} and if $(e_1,e_2,e_3,e_4)$ is the canonical basis of $\R^4$ then
\begin{align}
\label{def-nu-k}
&\nu_1=4e_1,\qquad \nu_2=4e_2,\qquad \nu_3=4e_3,\qquad \nu_4=4e_4,\\
&\nu_5=2e_1+2e_2,\qquad \nu_6=2e_3+2e_4,\qquad
\nu_7=2e_1+2e_3,\qquad \nu_8=2e_2+2e_4. \nonumber
\end{align}
We also set
\[
(\beta_k)_{k=1}^8
\ov{\mathrm{def}}{=}
(b_{11},b_{12},b_{21},b_{22},2u_1,2u_2,2v_1,2v_2).
\]
Now, we construct suitable weights $w_1,\ldots,w_8$. Suppose first that $x_j>0$ for any $1\leq j\leq4$ and write $x_j=\exp(y_j)$. Dividing \eqref{eq-scalar-polynomial} by 
$$
\exp(\la \alpha,y\ra)
\ov{\eqref{def-alpha-23}}{=}\exp(y_1)\exp(y_2)\exp(y_3)\exp(y_4)=x_1x_2x_3x_4 \ov{\eqref{def-alpha-23}}{=} x_{11}x_{12}x_{21}x_{22}
$$ 
gives
\begin{equation}
\label{def-de-grand-F}
F(y)
\ov{\mathrm{def}}{=}
\sum_{k=1}^8\beta_k\exp\big(\la\nu_k-\alpha,y\ra\big)
\geq K,
\qquad y=(y_1,y_2,y_3,y_4) \in \R^4.
\end{equation}

Since every monomial under consideration is homogeneous of degree four, the sum of the coordinates of each exponent vector $\nu_k$ is equal to $4$. On the other hand, the sum of the coordinates of $\alpha=(1,1,1,1)$ is also $4$. Hence, for any $1\leq k\leq8$,
\begin{equation}
\label{inter-ffj0}
\la \nu_k-\alpha,(1,1,1,1)\ra
=\sum_{j=1}^4(\nu_k)_j-\sum_{j=1}^4\alpha_j
=4-4
=0.
\end{equation}
Consequently, for any $y \in \R^4$ and any $t \in \R$, we have
\[
F(y+t(1,1,1,1))
\ov{\eqref{def-de-grand-F}}{=} \sum_{k=1}^8\beta_k\exp\big(\la\nu_k-\alpha,y+t(1,1,1,1)\ra\big)
\ov{\eqref{inter-ffj0}}{=} \sum_{k=1}^8\beta_k\exp\big(\la\nu_k-\alpha,y\ra\big)
\ov{\eqref{def-de-grand-F}}{=} F(y).
\]
It is therefore sufficient to minimize $F$ on the hyperplane
\[
H 
\ov{\mathrm{def}}{=} \left\{y \in \R^4 : y_1+y_2+y_3+y_4=0 \right\}.
\]
The restriction of $F$ to $H$ is coercive. Indeed, the first four summands of $F$ on $H$ are
\[
b_{11}\e^{4y_1},\qquad b_{12}\e^{4y_2},\qquad
b_{21}\e^{4y_3},\qquad b_{22}\e^{4y_4}.
\]
Moreover, if $y \in H$ and $\norm{y} \to \infty$, then $
\max_{1\leq j \leq4}y_j\to\infty$. Indeed, if one coordinate tends to $-\infty$, the condition $
y_1+y_2+y_3+y_4=0$ forces at least one of the other three coordinates to tend to $+\infty$, otherwise all four coordinates would remain bounded. One of the first four summands of $F(y)$ is therefore equal to $b\exp\left(4\max_{1\leq j\leq4}y_j\right)$ for some $b\in\{b_{11},b_{12},b_{21},b_{22}\}$. Since all the summands of $F$ are positive, we obtain 
\[
F(y)
\geq
\min_{1\leq i,j\leq2}b_{ij}
\exp\left(4\max_{1\leq j\leq4}y_j\right) \to \infty.
\]
It follows that the restriction of $F$ to $H$ attains its minimum at some point $y^0\in H$. Set
\begin{equation}
\label{def-y0}
S\ov{\mathrm{def}}{=}F(y^0)
\quad\text{and}\quad
w_k\ov{\mathrm{def}}{=}
\frac{\beta_k\exp\big(\langle\nu_k-\alpha,y^0\rangle\big)}{S},
\qquad 1\leq k\leq8.
\end{equation}
Thus $w_k$ is the proportion of the $k$th summand in the sum $F(y^0) \ov{\eqref{def-de-grand-F}}{=}\sum_{k=1}^8\beta_k\exp\big(\la\nu_k-\alpha,y^0\ra\big)$. In particular, we have
\begin{equation}
\label{numero-4}
w_k>0
\quad\text{and}\quad
\sum_{k=1}^8w_k=1.
\end{equation}
Since $y^0$ minimizes the restriction of $F$ to $H$, we have $\D F(y^0).h=0$ for any $h\in H$. Thus, for any $h \in H$, we have
\[
0
\ov{\eqref{def-de-grand-F}}{=} \sum_{k=1}^8 \beta_k
\exp\big(\la\nu_k-\alpha,y^0\ra\big)
\langle\nu_k-\alpha,h\rangle
=S\left\la \sum_{k=1}^8w_k(\nu_k-\alpha),h \right\ra
\]
Thus the vector $
\sum_{k=1}^8w_k(\nu_k-\alpha)$ is orthogonal to $H$. On the other hand, each vector $\nu_k-\alpha$ belongs to $H$, because the sum of its coordinates is zero. 
Since $H$ is a vector space, it follows that $
\sum_{k=1}^8w_k(\nu_k-\alpha)\in H$. This vector belongs both to $H$ and to $H^\perp$, and therefore it must be zero. Consequently, we have $\sum_{k=1}^8w_k(\nu_k-\alpha)=0$ or equivalently
\begin{equation}
\label{eq-weighted-exponent}
\sum_{k=1}^8 w_k\nu_k
= \alpha
\ov{\eqref{def-alpha-23}}{=} (1,1,1,1).
\end{equation}
In other words, the weighted exponent vector of the eight monomials is the exponent vector of $x_1x_2x_3x_4$. Finally, the definition of $w_k$ gives
\begin{equation}
\label{inter-12345R}
\frac{\beta_k}{w_k}
\ov{\eqref{def-y0}}{=} S\exp\big(-\la\nu_k-\alpha,y^0 \ra\big).
\end{equation}
Using $\sum_{k=1}^8 w_k \ov{\eqref{numero-4}}{=} 1$ and \eqref{eq-weighted-exponent}, we obtain
\begin{align*}
\MoveEqLeft
\prod_{k=1}^8 \left(\frac{\beta_k}{w_k}\right)^{w_k}
\ov{\eqref{inter-12345R}}{=}\prod_{k=1}^8\left(S\exp\big(-\langle\nu_k-\alpha,y^0\rangle\big)\right)^{w_k} 
=S^{\sum_{k=1}^8w_k}
\exp\left(-\left\la
\sum_{k=1}^8w_k(\nu_k-\alpha),y^0
\right\rangle\right) \\
&\ov{\eqref{numero-4}}{=} S\exp\left(-\left\la \sum_{k=1}^8 w_k\nu_k-\sum_{k=1}^8w_k\alpha,y^0 \right\ra\right)
\ov{\eqref{numero-4} \eqref{eq-weighted-exponent}}{=} S.
\end{align*}
Since $F(y) \ov{\eqref{def-de-grand-F}}{\geq} K$ for any $y \in \R^4$, we have $S \ov{\eqref{def-y0}}{=} F(y^0) \geq K$. Hence
\begin{equation}
\label{eq-coefficient-product}
\prod_{k=1}^8 \left(\frac{\beta_k}{w_k}\right)^{w_k}
\geq K.
\end{equation}
Consider the polytope of admissible weights
\[
\mathcal{W}
\ov{\mathrm{def}}{=}
\left\{
w\in\R_+^8:
\sum_{k=1}^8w_k=1
\text{ and }
\sum_{k=1}^8w_k\nu_k=\alpha
\right\}.
\]
Now, we determine its extreme points. Since $\alpha \ov{\eqref{def-alpha-23}}{=} (1,1,1,1)$, the equality $
\sum_{k=1}^8w_k\nu_k=\alpha$ is equivalent to
\begin{align}
\label{w-infty-inter}
4w_1+2w_5+2w_7& \ov{\eqref{def-nu-k}}{=} 1, \quad 4w_2+2w_5+2w_8 \ov{\eqref{def-nu-k}}{=}1,\\
4w_3+2w_6+2w_7&\ov{\eqref{def-nu-k}}{=} 1, \quad 4w_4+2w_6+2w_8 \ov{\eqref{def-nu-k}}{=}1.\label{w-infty-inter-bis}
\end{align}
Set $ a\ov{\mathrm{def}}{=}2w_5$, $b\ov{\mathrm{def}}{=}2w_6$, $c\ov{\mathrm{def}}{=}2w_7$, $d\ov{\mathrm{def}}{=}2w_8$. Then
\[
w_1 \ov{\eqref{w-infty-inter}}{=} \frac{1-a-c}{4},\qquad
w_2 \ov{\eqref{w-infty-inter}}{=} \frac{1-a-d}{4},\qquad
w_3 \ov{\eqref{w-infty-inter-bis}}{=} \frac{1-b-c}{4}\quad \text{and}\quad
w_4 \ov{\eqref{w-infty-inter-bis}}{=} \frac{1-b-d}{4}.
\]
Consequently, the nonnegativity of the eight weights is equivalent to
\[
a,b,c,d\geq0,
\qquad
a+c\leq1,\quad
a+d\leq1,\quad
b+c\leq1,\quad
b+d\leq1.
\]
Let $s\ov{\mathrm{def}}{=}\max\{a,b\}$ and $t\ov{\mathrm{def}}{=}\max\{c,d\}$. The four preceding inequalities imply $s+t\leq1$. If $s > 0$, put
\[
(u_1,u_2)
\ov{\mathrm{def}}{=}\left(\frac{a}{s},\frac{b}{s}\right)\in[0,1]^2,
\]
and take $(u_1,u_2)\ov{\mathrm{def}}{=}(0,0)$ if $s=0$. Define $(v_1,v_2)$ similarly from $(c,d)$ and $t$. Then, we have
\[
(a,b,c,d)
=s(u_1,u_2,0,0)
+t(0,0,v_1,v_2)
+(1-s-t)(0,0,0,0).
\]
Since $[0,1]^2$ is the convex hull of $
(0,0),(1,0),(0,1),(1,1)$, the admissible quadruples $(a,b,c,d)$ form the convex hull of
\[
0, e_1, e_2, e_3, e_4, e_1+e_2, e_3+e_4.
\]
Translating these seven points back into the variables $w_1,\ldots,w_8$ shows that the set $\mathcal{W}$ is the convex hull of the following seven points:
\[
\begin{split}
&(\tfrac14,\tfrac14,\tfrac14,\tfrac14,0,0,0,0),
\quad(\tfrac14,\tfrac14,0,0,0,\tfrac12,0,0),
\quad(\tfrac14,0,\tfrac14,0,0,0,0,\tfrac12),\\
&(0,\tfrac14,0,\tfrac14,0,0,\tfrac12,0),
\quad(0,0,0,0,0,0,\tfrac12,\tfrac12),
\quad(0,0,\tfrac14,\tfrac14,\tfrac12,0,0,0),\\
&(0,0,0,0,\tfrac12,\tfrac12,0,0).
\end{split}
\]
Let $m \ov{\mathrm{def}}{=} (q_{11},q_{12},q_{21},q_{22},r_1,r_2,c_1,c_2)$. By H\"older's inequality and \eqref{def-qij-ri}, we have
\begin{equation}
\label{eq-tau-basic}
|\tau|\leq(q_{11}q_{12}q_{21}q_{22})^{\frac{1}{4}},
\quad |\tau|\leq(r_1r_2)^{\frac{1}{2}},
\quad |\tau|\leq(c_1c_2)^{\frac{1}{2}},
\end{equation}
and
\begin{equation}
\label{eq-edge-bounds}
r_1\leq(q_{11}q_{12})^{\frac{1}{2}},\quad r_2\leq(q_{21}q_{22})^{\frac{1}{2}},\quad c_1\leq(q_{11}q_{21})^{\frac{1}{2}},\quad c_2\leq(q_{12}q_{22})^{\frac{1}{2}}.
\end{equation}
If $\tau=0$, inequality \eqref{eq-operator-polynomial} is immediate, since every term on its left-hand side is nonnegative. We may therefore suppose that $\tau\neq0$. It follows from \eqref{eq-tau-basic} that all the entries of
\[
m=(q_{11},q_{12},q_{21},q_{22},r_1,r_2,c_1,c_2)
\]
are strictly positive. Denote the seven extreme points of $\mathcal{W}$ displayed previously by $w^{(1)},\ldots,w^{(7)}$. Recall that
\[
m=(m_1,\ldots,m_8)
=(q_{11},q_{12},q_{21},q_{22},r_1,r_2,c_1,c_2).
\]
For each extreme point $w^{(s)}$, consider the weighted product $
\prod_{k=1}^8m_k^{w_k^{(s)}}$. For instance, for $
w^{(1)}
=
\left(\frac14,\frac14,\frac14,\frac14,0,0,0,0\right)$, we obtain $
\prod_{k=1}^8m_k^{w_k^{(1)}}
=
(q_{11}q_{12}q_{21}q_{22})^{\frac14}$, 
whereas for $
w^{(2)}
=
\left(\frac14,\frac14,0,0,0,\frac12,0,0\right)$, we obtain $
\prod_{k=1}^8m_k^{w_k^{(2)}}
=
(q_{11}q_{12})^{\frac14}r_2^{\frac12}$. Applying the same substitution to the remaining five extreme points gives the seven quantities
\[
\begin{gathered}
(q_{11}q_{12}q_{21}q_{22})^{\frac14},\qquad
(q_{11}q_{12})^{\frac14}r_2^{\frac12},\qquad
(q_{11}q_{21})^{\frac14}c_2^{\frac12},\\
(q_{12}q_{22})^{\frac14}c_1^{\frac12},\qquad
(c_1c_2)^{\frac12},\qquad
(q_{21}q_{22})^{\frac14}r_1^{\frac12},\qquad
(r_1r_2)^{\frac12}.
\end{gathered}
\]

The first, fifth and seventh quantities dominate $|\tau|$ directly by
\eqref{eq-tau-basic}. For the second one, we have
\[
|\tau|
\leq(r_1r_2)^{\frac12}
\ov{\eqref{eq-edge-bounds}}{\leq} \left((q_{11}q_{12})^{\frac12}r_2\right)^{\frac12}
=(q_{11}q_{12})^{\frac14}r_2^{\frac12}.
\]
The sixth inequality follows in the same way using the estimate for $r_2$ in
\eqref{eq-edge-bounds}. Similarly,
\[
|\tau|
\ov{\eqref{eq-tau-basic}}{\leq} (c_1c_2)^{\frac12}
\ov{\eqref{eq-edge-bounds}}{\leq} (q_{11}q_{21})^{\frac14}c_2^{\frac12},
\]
and the fourth inequality follows using the estimate for $c_2$. 
Let $w \in \cal{W}$. Since $\mathcal{W}$ is the convex hull of these seven points, there exist numbers $\theta_1,\ldots,\theta_7\geq0$ satisfying
\begin{equation}
\label{inter-ytu-77}
\sum_{s=1}^7\theta_s=1
\quad\text{and}\quad
w=\sum_{s=1}^7\theta_sw^{(s)}.
\end{equation}
For each extreme point $w^{(s)}$, the preceding estimates show that
\begin{equation}
\label{ine-inter-3456}
\prod_{k=1}^8m_k^{w_k^{(s)}}
\geq|\tau|.
\end{equation}
Therefore
\begin{align}
\MoveEqLeft
\label{inter-3465}
\prod_{k=1}^8m_k^{w_k}
\ov{\eqref{inter-ytu-77}}{=} \prod_{k=1}^8m_k^{\sum_{s=1}^7\theta_sw_k^{(s)}}
= \prod_{s=1}^7\left(\prod_{k=1}^8m_k^{w_k^{(s)}}\right)^{\theta_s}
\ov{\eqref{ine-inter-3456}}{\geq} \prod_{s=1}^7|\tau|^{\theta_s}
=|\tau|^{\sum_{s=1}^7 \theta_s}
\ov{\eqref{inter-ytu-77}}{=}|\tau|.         
\end{align}
Recall that the weights $w_1,\ldots,w_8$ constructed previously are strictly positive (see \eqref{numero-4}). The weighted arithmetic-geometric mean inequality\footnote{\thefootnote. For positive real numbers $x_1, x_2, \dots, x_n$ and non-negative weights $w_1, w_2, \dots, w_n$ such that $\sum_{i=1}^{n} w_i = 1$, the weighted arithmetic-geometric mean inequality states that:
\begin{equation}
\sum_{i=1}^{n} w_i x_i 
\geq \prod_{i=1}^{n} x_i^{w_i}.
\end{equation}
}, \eqref{eq-coefficient-product} and the preceding estimate give
\begin{align*}
\MoveEqLeft
\sum_{k=1}^8\beta_km_k
=\sum_{k=1}^8w_k\frac{\beta_km_k}{w_k}
\geq\prod_{k=1}^8
\left(\frac{\beta_km_k}{w_k}\right)^{w_k}
=
\left(
\prod_{k=1}^8
\left(\frac{\beta_k}{w_k}\right)^{w_k}
\right)
\left(
\prod_{k=1}^8m_k^{w_k}
\right)
\ov{\eqref{eq-coefficient-product}\eqref{inter-3465}}{\geq} K|\tau|.         
\end{align*}
This proves \eqref{eq-operator-polynomial}.
\end{proof}

\begin{thm}
\label{thm-M2}
For any matrix $C \in \M_2$, the Schur multiplier $M_C \co S^4_2 \to S^4_2$ satisfies
\[
\norm{M_C}_{\cb,S^4_2\to S^4_2}
=\norm{M_C}_{S^4_2 \to S^4_2}.
\]
Consequently, the symbol in Theorem~\ref{thm-main} has the smallest possible square size.
\end{thm}

\begin{proof}
Write $C=[c_{ij}]_{i,j=1}^2$ and set $\lambda \ov{\mathrm{def}}{=} \norm{M_C}_{S^4_2 \to S^4_2}^4$. Fix $\mu>\lambda$, an integer $m \geq 1$, and a block matrix $X=[X_{ij}]_{i,j=1}^2$ in $\M_2(\M_m)$. Set
\begin{equation}
\label{def-de-P-Q-R}
P \ov{\mathrm{def}}{=} X_{11}^*X_{11}+X_{21}^*X_{21},\qquad
Q \ov{\mathrm{def}}{=} X_{12}^*X_{12}+X_{22}^*X_{22},\qquad
R \ov{\mathrm{def}}{=} X_{11}^*X_{12}+X_{21}^*X_{22}.
\end{equation}
Then
\begin{equation}
\label{inter-28}
X^*X
=\begin{bmatrix}
  X_{11}   & X_{12}  \\
  X_{21}   &  X_{22} \\
\end{bmatrix}^*
\begin{bmatrix}
  X_{11}    & X_{12}  \\
   X_{21}  &  X_{22} \\
\end{bmatrix}
=\begin{bmatrix}
  X_{11}^*   & X_{21}^*  \\
  X_{12}^*   &  X_{22}^* \\
\end{bmatrix}
\begin{bmatrix}
  X_{11}  & X_{12}  \\
  X_{21}  &  X_{22} \\
\end{bmatrix}
\ov{\eqref{def-de-P-Q-R}}{=}\begin{bmatrix}
P&R\\
R^*&Q
\end{bmatrix},
\end{equation}
Hence $(X^*X)^2 \ov{\eqref{inter-28}}{=}\begin{bmatrix}
  P^2+RR^*   & PR+RQ  \\
 R^*P+QR^*    &  R^*R+Q^2 \\
\end{bmatrix}$. Consequently, we have
\begin{equation}
\label{inter-34567}
\tr|X|^4
=\tr((X^*X)^2)
= \tr(P^2)+\tr(Q^2)+2\tr(R^*R).
\end{equation}
By expanding the two first terms and using the cyclicity of the trace, we obtain
\begin{align}
\MoveEqLeft
\label{inter-ABC-1}         
\tr(P^2)
\ov{\eqref{def-de-P-Q-R}}{=} \tr\big[(X_{11}^*X_{11}+X_{21}^*X_{21})(X_{11}^*X_{11}+X_{21}^*X_{21})\big] \\
&=\tr\big[X_{11}^*X_{11}X_{11}^*X_{11}+X_{11}^*X_{11}X_{21}^*X_{21} +X_{21}^*X_{21}X_{11}^*X_{11} +X_{21}^*X_{21}X_{21}^*X_{21} \big] \nonumber\\
&=q_{11}+q_{21}+2c_1, \nonumber
\end{align}
where $q_{11} \ov{\mathrm{def}}{=} \tr(X_{11}^*X_{11}X_{11}^*X_{11})$, $q_{21} \ov{\mathrm{def}}{=} \tr(X_{21}^*X_{21}X_{21}^*X_{21})$, $c_1 \ov{\mathrm{def}}{=} \tr(X_{11}^*X_{11}X_{21}^*X_{21})$ and
\begin{align}
\MoveEqLeft
\label{inter-ABC-2}         
\tr(Q^2)
\ov{\eqref{def-de-P-Q-R}}{=} \tr\big[(X_{12}^*X_{12}+X_{22}^*X_{22})(X_{12}^*X_{12}+X_{22}^*X_{22})\big]\\
&=\tr\big[X_{12}^*X_{12}X_{12}^*X_{12} +X_{12}^*X_{12}X_{22}^*X_{22} +X_{22}^*X_{22}X_{12}^*X_{12} +X_{22}^*X_{22}X_{22}^*X_{22}\big] \nonumber\\
&=q_{12}+q_{22}+2c_2, \nonumber
\end{align}
where $q_{12} \ov{\mathrm{def}}{=} \tr(X_{12}^*X_{12}X_{12}^*X_{12})$, $q_{22} \ov{\mathrm{def}}{=} \tr(X_{22}^*X_{22}X_{22}^*X_{22})$, $c_2 \ov{\mathrm{def}}{=} \tr(X_{12}^*X_{12}X_{22}^*X_{22})$. The third term of \eqref{inter-34567} gives
\begin{align}
\MoveEqLeft
\label{inter-BCD}
\tr(R^*R)
\ov{\eqref{def-de-P-Q-R}}{=} \tr\big[(X_{11}^*X_{12}+X_{21}^*X_{22})^*(X_{11}^*X_{12}+X_{21}^*X_{22})\big] \\
&=\tr\big[(X_{12}^*X_{11}+X_{22}^*X_{21})(X_{11}^*X_{12}+X_{21}^*X_{22})\big] \nonumber\\
&=  \tr\big[X_{12}^*X_{11}X_{11}^*X_{12} +X_{12}^*X_{11}X_{21}^*X_{22} +X_{22}^*X_{21}X_{11}^*X_{12} +X_{22}^*X_{21}X_{21}^*X_{22}  \big]  \nonumber \\
&=r_1+r_2+2\Re\tr(X_{12}^*X_{11}X_{21}^*X_{22})
=r_1+r_2+2\Re\tau, \nonumber
\end{align}
where $r_1\ov{\mathrm{def}}{=} \tr(X_{12}^*X_{11}X_{11}^*X_{12})$, $r_2\ov{\mathrm{def}}{=} \tr(X_{22}^*X_{21}X_{21}^*X_{22})$ and $\tau \ov{\mathrm{def}}{=} \tr(X_{12}^*X_{11}X_{21}^*X_{22})$. Therefore a direct expansion gives
\begin{equation}
\label{eq-block-expansion}
\norm{X}_{S_2^4(S_m^4)}^4
\ov{\eqref{inter-34567}\eqref{inter-ABC-1}\eqref{inter-ABC-2}\eqref{inter-BCD}}{=} \tr|X|^4
=\sum_{i,j=1}^2 q_{ij}+2\sum_{i=1}^2 r_i+2\sum_{j=1}^2 c_j+4\Re\tau.
\end{equation}
Set
\begin{equation}
\label{inter-3456}
\gamma\ov{\mathrm{def}}{=} c_{11}\ovl{c_{12}}\ovl{c_{21}}c_{22},
\quad d_{ij}\ov{\mathrm{def}}{=} \mu-|c_{ij}|^4,
\quad \rho_i\ov{\mathrm{def}}{=} \mu-|c_{i1}c_{i2}|^2,
\quad \kappa_j \ov{\mathrm{def}}{=} \mu-|c_{1j}c_{2j}|^2,
\quad \eta \ov{\mathrm{def}}{=} \mu-\gamma.
\end{equation}
Testing the Schur multiplier $M_C \co S^4_2 \to S^4_2$ on matrix units gives $\lambda=\norm{M_C}_{S^4_2 \to S^4_2}^4 \geq \max_{1 \leq i,j \leq 2} |c_{ij}|^4$. Moreover, we have
\[
|c_{i1}c_{i2}|^2
\leq\max\{|c_{i1}|^4,|c_{i2}|^4\}
\leq\lambda
\]
and, similarly,
\[
|c_{1j}c_{2j}|^2
\leq\max\{|c_{1j}|^4,|c_{2j}|^4\}
\leq\lambda.
\]
Therefore $d_{ij}$, $\rho_i$ and $\kappa_j$ are strictly positive since $\mu>\lambda$. We will apply \eqref{eq-block-expansion} before and after the multiplier. If $Y \ov{\mathrm{def}}{=}(\Id_{S_m^4}\ot M_C)(X)$, then $Y_{ij}=c_{ij}X_{ij}$ for any $1 \leq i,j \leq 2$. Hence the quantities associated with $Y$ instead of $X$ in \eqref{eq-block-expansion} are
\begin{equation}
\label{inter-ergbkl}
q_{ij}(Y)=|c_{ij}|^4q_{ij},\qquad
r_i(Y)=|c_{i1}c_{i2}|^2r_i,\qquad
c_j(Y)=|c_{1j}c_{2j}|^2c_j
\end{equation}
and
\begin{equation}
\label{inter-456787}
\tau(Y)
=\tr(Y_{12}^*Y_{11}Y_{21}^*Y_{22})
=c_{11}\overline{c_{12}}\overline{c_{21}}c_{22}\tau
\ov{\eqref{inter-3456}}{=}\gamma\tau.
\end{equation}
Consequently, we have
\begin{equation}
\label{inter-788YU}
\norm{Y}_{S_2^4(S_m^4)}^4
\ov{\eqref{eq-block-expansion}\eqref{inter-ergbkl} \eqref{inter-456787}}{=} \sum_{i,j=1}^2|c_{ij}|^4q_{ij}
+2\sum_{i=1}^2|c_{i1}c_{i2}|^2r_i
+2\sum_{j=1}^2|c_{1j}c_{2j}|^2c_j
+4\Re(\gamma\tau).
\end{equation}
Subtracting this identity from $\mu$ times \eqref{eq-block-expansion} therefore gives
\begin{align}
\MoveEqLeft
\label{eq-amplified-difference}
\mu\norm{X}_{S_2^4(S_m^4)}^4-\norm{(\Id_{S^4_m}\ot M_C)(X)}_{S_2^4(S_m^4)}^4 \\
&\ov{\eqref{eq-block-expansion} \eqref{inter-788YU}}{=} \sum_{i,j=1}^2 \mu q_{ij}+2\sum_{i=1}^2 \mu r_i+2\sum_{j=1}^2 \mu c_j+4\mu\Re\tau
 \nonumber \\
&- \sum_{i,j=1}^2|c_{ij}|^4q_{ij}
-2\sum_{i=1}^2|c_{i1}c_{i2}|^2r_i
-2\sum_{j=1}^2|c_{1j}c_{2j}|^2c_j
-4\Re(\gamma\tau) \nonumber\\
&\ov{\eqref{inter-3456}}{=}\sum_{i,j=1}^2 d_{ij}q_{ij}+2\sum_{i=1}^2 \rho_ir_i+2\sum_{j=1}^2\kappa_jc_j+4\Re(\eta\tau). \nonumber
\end{align}
Given arbitrary $x_{ij} \geq 0$, choose complex scalars $z_{ij}$ with $|z_{ij}| = x_{ij}$ and with the phase of $\overline{z_{12}}z_{11}\overline{z_{21}}z_{22}$ chosen so that
\begin{equation}
\label{inter-45}
\Re(\eta\overline{z_{12}}z_{11}\overline{z_{21}}z_{22}) 
= -|\eta|x_{11}x_{12}x_{21}x_{22}.
\end{equation}
Let $z=[z_{ij}]_{i,j=1}^2 \in S^4_2$. By the definition of $\lambda$, we have
\begin{equation}
\label{inter-29-28}
\norm{M_C(z)}_{S_2^4}^4
\leq\lambda\norm{z}_{S_2^4}^4
<\mu\norm{z}_{S_2^4}^4
\end{equation}
whenever $z\neq0$. Applying \eqref{eq-block-expansion} to the scalar matrix $z$, for which
\[
q_{ij}=x_{ij}^4,\qquad
r_i=x_{i1}^2x_{i2}^2,\qquad
c_j=x_{1j}^2x_{2j}^2,\qquad
\tau=\overline{z_{12}}z_{11}\overline{z_{21}}z_{22},
\]
we obtain
\[
\begin{split}
0
&\ov{\eqref{inter-29-28}}{\leq} \mu\norm{z}_{S_2^4}^4-\norm{M_C(z)}_{S_2^4}^4\\
&=\sum_{i,j=1}^2d_{ij}x_{ij}^4
+2\sum_{i=1}^2\rho_ix_{i1}^2x_{i2}^2
+2\sum_{j=1}^2\kappa_jx_{1j}^2x_{2j}^2
+4\Re\big(\eta\overline{z_{12}}z_{11}\overline{z_{21}}z_{22}\big)\\
&\ov{\eqref{inter-45}}{=}\sum_{i,j=1}^2d_{ij}x_{ij}^4
+2\sum_{i=1}^2\rho_ix_{i1}^2x_{i2}^2
+2\sum_{j=1}^2\kappa_jx_{1j}^2x_{2j}^2
-4|\eta|x_{11}x_{12}x_{21}x_{22}.
\end{split}
\]
Rearranging this inequality gives
\[
\sum_{i,j=1}^2 d_{ij}x_{ij}^4+2\sum_{i=1}^2\rho_ix_{i1}^2x_{i2}^2+2\sum_{j=1}^2\kappa_jx_{1j}^2x_{2j}^2
\geq 4|\eta|x_{11}x_{12}x_{21}x_{22}.
\]
Lemma~\ref{lem-lifting} therefore implies
\begin{equation}
\label{inter-39}
\sum_{i,j=1}^2 d_{ij}q_{ij}+2\sum_{i=1}^2 \rho_i r_i+2\sum_{j=1}^2\kappa_jc_j
\ov{\eqref{eq-operator-polynomial}}{\geq} 4|\eta||\tau|.
\end{equation}
Inserting this into \eqref{eq-amplified-difference}, we obtain
\[
\mu\|X\|_{S_m^4(S_2^4)}^4-\|(\Id_{S_m^4}\ot M_C)(X)\|_{S_m^4(S_2^4)}^4
\ov{\eqref{inter-39} \eqref{eq-amplified-difference}}{\geq} 4|\eta||\tau|+4\Re(\eta\tau)
\geq 0.
\]
Thus $\norm{\Id_{S_m^4} \ot M_C}_{S_m^4(S_2^4) \to S_m^4(S_2^4)}^4 \leq \mu$ for any integer $m \geq 1$. Letting $\mu$ decrease to $\lambda$ proves 
$$
\|M_C\|_{\cb,S^4_2 \to S^4_2}\leq \norm{M_C}_{S^4_2 \to S^4_2}.
$$ 
The reverse inequality is immediate from the first matrix level.
\end{proof}

\section{A variational formula for two-column Schur multipliers}
\label{sec-variational}

Now, we develop the main structural tool used to estimate the ordinary norm. Rather than optimizing directly over rectangular matrices, we transform the $S^4$ multiplier norm of an arbitrary two-column symbol into an optimization problem involving only positive $2\times2$ matrices. The reduction uses the structure of the $S^4$ norm, convexity of finite positive operator-valued measures, and conic duality. The resulting variational formula is independent of the particular symbol considered in Theorem~\ref{thm-main} and will subsequently reduce its scalar norm computation to a low-dimensional problem.

Let $m \geq 2$. Consider a matrix $B=[b_{kj}]$ in $\M_{m,2}$. For any integer $1 \leq k \leq m$, we introduce the matrix $D_k \ov{\mathrm{def}}{=}  \operatorname{diag}(b_{k1},b_{k2})$ in $\M_2$. For any positive matrix $Q \in \M_2$, define
\begin{equation}
\label{eq-fQ}
f_B(Q)
\ov{\mathrm{def}}{=} \inf\bigl\{\|W\|_{S^2_2} : W=W^*, W\geq D_kQD_k^*\text{ for }1 \leq k \leq m\bigr\}.
\end{equation}
The admissible set in \eqref{eq-fQ} is nonempty, since the matrix $t\I_2$ is admissible for every sufficiently large $t$. In particular, $f_B(Q)$ is finite. 


We briefly recall the terminology from convex duality that will be used below. We refer to \cite{BoV04} , \cite{Gul10}, \cite{Lue97}, \cite{NoW06} and \cite{WSV00}. Following \cite[Definition 4.17 p.~93]{Gul10}, a subset $C$ of a vector space is called a cone if $t x \in C$ whenever $t > 0$ and $x \in C$. Let $E$ be a finite-dimensional real Hilbert space. A closed, convex cone in $E$ with a nonempty interior and containing no whole lines is called a regular convex cone. We will use the notation
$$
K^* 
\ov{\mathrm{def}}{=} \{z \in E : \la x, z \ra \geq 0 \text{ for all } x \in K\}
$$
of \cite[p.~295]{Gul10} for the (modified) dual cone, which is the reflection through the origin of the usual dual cone.

%

Now, we prove the following variational formula.

\begin{prop}
\label{prop-variational}
For any matrix $B \in \M_{m,2}$, we have
\begin{equation}
\label{eq-variational}
\norm{M_B}_{S^4_{m,2} \to S^4_{m,2}}^2
=\sup_{\substack{Q \geq 0 \\ \|Q\|_{S^2_2}=1} }f_B(Q).
\end{equation}
\end{prop}

\begin{proof}
Let $X \in \M_{m,2}$. Write the rows of $X$ as $x_1,\ldots,x_m$ and put $P_k\ov{\mathrm{def}}{=}x_k^*x_k$ for any integer $1\leq k\leq m$. Since
\begin{equation}
\label{inter-GHJ789}
X^*X
=\sum_{k=1}^m x_k^*x_k
=\sum_{k=1}^mP_k,
\end{equation}
we have
\[
\norm{X}_{S^4_{m,2}}^4
= \tr((X^*X)^2)
\ov{\eqref{inter-GHJ789}}{=}\tr\left(\left(\sum_{k=1}^mP_k\right)^2\right)
=\norm{\sum_{k=1}^mP_k}_{S^2_2}^2.
\]
Moreover, the $k$th row of $B\circ X$ is $x_kD_k=x_k\begin{bmatrix}
    b_{k1} & 0  \\
   0  & b_{k2}  \\
\end{bmatrix}$. Consequently, we have
\begin{equation}
\label{inter-29000-21}
(B\circ X)^*(B\circ X)
=\sum_{k=1}^m(x_kD_k)^*(x_kD_k)
=\sum_{k=1}^m D_k^*x_k^*(x_kD_k)
=\sum_{k=1}^m D_k^*P_kD_k,
\end{equation}
and therefore
\[
\norm{B \circ X}_{S^4_{m,2}}^4
=\tr\big[(B \circ X)^*(B \circ X)\big]^2
\ov{\eqref{inter-29000-21}}{=} \tr\bigg[\sum_{k=1}^m D_k^*P_kD_k\bigg]^2
=\norm{\sum_{k=1}^m D_k^*P_k D_k}_{S^2_2}^2.
\]
Conversely, every positive matrix in $\M_2$ of rank at most one can be written as $x^*x$ for some $x\in\M_{1,2}$. It follows that
\begin{equation}
\label{eq-positive-decomposition-rank-one}
\norm{M_B}_{S^4_{m,2}\to S^4_{m,2}}^2
=\sup_{\substack{P_k\geq0,\ \operatorname{rank}(P_k)\leq1\\
\sum_{k=1}^mP_k\neq0}}
\frac{\norm{\sum_{k=1}^mD_k^*P_kD_k}_{S^2_2}}
{\norm{\sum_{k=1}^mP_k}_{S^2_2}}.
\end{equation}
Now, we show that the rank conditions in \eqref{eq-positive-decomposition-rank-one} can be omitted. Fix a nonzero positive matrix $H\in\M_2$ and consider the compact convex set
\[
\mathcal{P}(H)
\ov{\mathrm{def}}{=}
\left\{(P_1,\ldots,P_m):P_k\geq0\text{ and }\sum_{k=1}^mP_k=H\right\}.
\]
Suppose first that $H$ is invertible. The change of variables $
P_k=H^{\frac12}E_kH^{\frac12}$ identifies $\mathcal{P}(H)$ with the set of finite positive operator-valued measures on $\mathbb{C}^2$. More precisely, a finite positive operator-valued measure \cite[Definition 2.34 p.~101]{Wat18} with $m$ outcomes is a family $(E_1,\ldots,E_m)$ of positive matrices in $\M_2$ satisfying $
\sum_{k=1}^mE_k=\I_2$. Indeed, the preceding change of variables gives
\[
E_k=H^{-\frac12}P_kH^{-\frac12}\geq0
\quad\text{and}\quad
\sum_{k=1}^mE_k
=H^{-\frac12}\left(\sum_{k=1}^mP_k\right)H^{-\frac12}
=\I_2.
\]
For fixed $H$, define
\begin{equation}
\label{def-Phi-H}
\Phi_H(E_1,\ldots,E_m)
\ov{\mathrm{def}}{=}
\norm{\sum_{k=1}^mD_k^*H^{\frac12}E_kH^{\frac12}D_k}_{S^2_2}.
\end{equation}
This function is convex. Indeed, if $(E_k)$ and $(F_k)$ are two positive operator-valued measures and $0\leq t\leq1$, then the triangle inequality gives
\[
\begin{split}
\Phi_H(t(E_k)+(1-t)(F_k))
&\ov{\eqref{def-Phi-H}}{=} \norm{t\sum_{k=1}^mD_k^*H^{\frac12}E_kH^{\frac12}D_k
+(1-t)\sum_{k=1}^mD_k^*H^{\frac12}F_kH^{\frac12}D_k}_{S^2_2}\\
&\leq t\norm{\sum_{k=1}^mD_k^*H^{\frac12}E_kH^{\frac12}D_k}_{S^2_2}
+(1-t)\norm{\sum_{k=1}^mD_k^*H^{\frac12}F_kH^{\frac12}D_k}_{S^2_2} \\
&\ov{\eqref{def-Phi-H}}{=} t\Phi_H(E_1,\ldots,E_m)+(1-t)\Phi_H(F_1,\ldots,F_m).
\end{split}
\]
Since the set of positive operator-valued measures is compact and convex by \cite[pp.~103-104]{Wat18}, $\Phi_H$ attains its maximum at an extreme point of this set according to \cite[7.69 p.~298]{AB06}. 
Let $(E_1,\ldots,E_m)$ be such an extreme point. Put $r_k\ov{\mathrm{def}}{=}\operatorname{rank}(E_k)$.  
By \cite[Corollary 2.1 p.~564]{Par99}, we have the inequality
\begin{equation}
\label{eq-rank-extreme-POVM}
\sum_{k=1}^m r_k^2\leq4.
\end{equation}
If no $E_k$ has rank two, then every nonzero $E_k$ has rank one. Consequently, we have
\[
\operatorname{rank}(P_k)
=\operatorname{rank}\big(H^{\frac12}E_kH^{\frac12}\big)
\leq1.
\]
If some $E_{k_0}$ has rank two, then \eqref{eq-rank-extreme-POVM} implies that $E_k=0$ for every $k\neq k_0$. Since $\sum_{k=1}^mE_k=\I_2$, we necessarily have $E_{k_0}=\I_2$. In this case, $P_{k_0}=H$ and $P_k=0$ for $k\neq k_0$. The corresponding quotient satisfies
\[
\frac{\norm{D_{k_0}^*HD_{k_0}}_{S^2_2}}{\norm{H}_{S^2_2}}
\leq\norm{D_{k_0}}_{\M_2}^2.
\]
Let $\xi \in \mathbb{C}^2$ be a unit vector such that $\norm{D_{k_0}^*\xi}_2=\norm{D_{k_0}}_{\M_2}$ and set $P_{k_0}' \ov{\mathrm{def}}{=}\xi\xi^*$ and $P_k'\ov{\mathrm{def}}{=} 0$ for $k\neq k_0$. Notice that the rank-one family $(P_k')$ constructed previously need not have sum $H$. This causes no difficulty because the quotient in \eqref{eq-positive-decomposition-rank-one} is optimized over all nonzero families, and hence over all possible sums $H$. We have shown that every quotient associated with an arbitrary positive family is bounded above by a quotient associated with a rank-one family. The converse inequality is immediate because rank-one families form a subclass of all positive families. Then
\[
\frac{\norm{D_{k_0}^*P_{k_0}'D_{k_0}}_{S^2_2}}
{\norm{P_{k_0}'}_{S^2_2}}
=\norm{D_{k_0}}_{\M_2}^2.
\]
Thus the quotient associated with the exceptional extreme point $(0,\ldots,\I_2,\ldots,0)$ is no larger than a quotient obtained with rank-one matrices.

Finally, if $H$ is singular and nonzero, then $\operatorname{rank}(H)=1$. Since $0\leq P_k\leq H$ for every $k$, the range of $P_k$ is contained in the range of $H$. Hence $\operatorname{rank}(P_k)\leq1$. We conclude that the rank conditions in \eqref{eq-positive-decomposition-rank-one} can be omitted. Therefore
\begin{equation}
\label{eq-positive-decomposition}
\norm{M_B}_{S^4_{m,2}\to S^4_{m,2}}^2
=\sup_{\substack{P_k\geq0\\\sum_{k=1}^mP_k\neq0}}
\frac{\norm{\sum_{k=1}^mD_k^*P_kD_k}_{S^2_2}}
{\norm{\sum_{k=1}^mP_k}_{S^2_2}}.
\end{equation}
Now, we dualize the numerator. If $T \in \M_2$ is positive, Hilbert--Schmidt duality gives
\begin{equation}
\label{Hilbert-Schmidt-duality}
\norm{T}_{S^2_2}
=\sup_{\substack{Q\geq0\\\norm{Q}_{S^2_2}=1}}\tr(QT).
\end{equation}
Applying this identity to $T=\sum_{k=1}^mD_k^*P_kD_k$ and using the cyclicity of the trace, we obtain
\begin{equation}
\label{inter-sdfghyy}
\norm{\sum_{k=1}^mD_k^*P_kD_k}_{S^2_2}
\ov{\eqref{Hilbert-Schmidt-duality}}{=} \sup_{\substack{Q\geq0\\\norm{Q}_{S^2_2}=1}}
\sum_{k=1}^m\tr(QD_k^*P_kD_k)
=\sup_{\substack{Q\geq0\\\norm{Q}_{S^2_2}=1}}
\sum_{k=1}^m\tr(D_kQD_k^*P_k).
\end{equation}
Since two suprema may be interchanged, \eqref{eq-positive-decomposition} gives
\begin{equation}
\label{eq-double-supremum}
\norm{M_B}_{S^4_{m,2}\to S^4_{m,2}}^2
\ov{\eqref{eq-positive-decomposition}\eqref{inter-sdfghyy} }{=}\sup_{\substack{Q\geq0\\\norm{Q}_{S^2_2}=1}}
\sup_{\substack{P_k\geq0\\\sum_{k=1}^mP_k\neq0}}
\frac{\sum_{k=1}^m\tr(D_kQD_k^*P_k)}
{\norm{\sum_{k=1}^mP_k}_{S^2_2}}.
\end{equation}
Fix such a matrix $Q$ and set $A_k\ov{\mathrm{def}}{=}D_kQD_k^*$. By homogeneity, the inner supremum in \eqref{eq-double-supremum} is equal to the value of the conic optimization problem
\begin{equation}
\label{eq-primal-conic}
\sup\left\{
\sum_{k=1}^m\tr(A_kP_k):
P_k\geq0,\ R=\sum_{k=1}^mP_k, \norm{R}_{S^2_2}\leq1
\right\}.
\end{equation}
Notice that the condition $R\geq0$ is automatic. Now, we put this problem precisely into the conic-programming framework of \cite[Theorem~11.23 p.~295]{Gul10}. Let $
\mathcal H\ov{\mathrm{def}}{=}\M_2^{\mathrm{sa}}$, viewed as a real Euclidean space with inner product $
\la U,V\ra=\tr(UV)$, and let $
\mathcal H_+
\ov{\mathrm{def}}{=}
\{P\in\mathcal H:P\geq0\}$. We also introduce the Lorentz cone 
\[
\mathcal L
\ov{\mathrm{def}}{=}
\big\{
(R,t) \in \cal{H} \times \R : \norm{R}_{S^2_2}\leq t \big\}
\]
over $\cal{H}$. Consider the vector space $
E \ov{\mathrm{def}}{=} \cal{H}^m \times \cal{H} \times \R$ and  the cone $K \ov{\mathrm{def}}{=} \cal{H}_+^m\times \mathcal L$. Both $\mathcal H_+$ and $\mathcal L$ are self-dual regular convex cones, and hence so is their product $K$. In particular, we have
\[
K^*=(\mathcal H_+^*)^m \times \mathcal L^*
=\mathcal H_+^m\times\mathcal L
=K.
\] 
For any $x=(P_1,\ldots,P_m,R,t) \in E$, we define the element 
\begin{equation}
\label{def-A-x}
\mathscr A x
\ov{\mathrm{def}}{=}
\bigg(R-\sum_{k=1}^m P_k,t\bigg)
\end{equation}
of the space $\cal{H} \times \R$ and we consdider the element $
b \ov{\mathrm{def}}{=}v(0,1)$ in the space $\cal{H} \times \R$ and we put $c \ov{\mathrm{def}}{=} (-A_1,\ldots,-A_m,0,0)$. Then the negative of \eqref{eq-primal-conic} is exactly the conic program
\begin{equation}
\label{primal-conic}
\inf\left\{ \la c,x \ra : \mathscr{A} x=b,\quad x \in K \right\},
\end{equation}
which is of the form considered in \cite[Theorem~11.23, p.~295]{Gul10}. Now, we describe the adjoint of the operator $\mathscr{A} \co E \to \cal{H} \times \R$. Let $(W,\alpha) \in \cal{H} \times \R$ and $x=(P_1,\ldots,P_m,R,t) \in E$. Since
\begin{align*}
\MoveEqLeft
\la\mathscr{A} x,(W,\alpha) \ra_{\cal{H} \times \R}
\ov{\eqref{def-A-x}}{=} \bigg\la \bigg(R-\sum_{k=1}^m P_k,t\bigg),(W,\alpha) \bigg\ra
=\tr\left(W\left(R-\sum_{k=1}^mP_k\right)\right)+\alpha t \\
&= \sum_{k=1}^m\tr((-W)P_k)+\tr(WR)+\alpha t 
=\big\la (P_1,\ldots,P_m,R,t), (-W,\ldots,-W,W,\alpha) \big\ra_E \\
&=\big\la x, (-W,\ldots,-W,W,\alpha) \big\ra_E,
\end{align*}
we have
\[
\mathscr{A}^*(W,\alpha)
=(-W,\ldots,-W,W,\alpha).
\]
Consequently, we have
\[
c-\mathscr A^*(W,\alpha)
=
(W-A_1,\ldots,W-A_m,-W,-\alpha).
\]
Using the self-duality of $K$, the dual conic constraint $
c-\mathscr A^*(W,\alpha) \in K^*$ of \cite[Theorem~11.23, p.~295]{Gul10} is equivalent to
\[
W\geq A_k
\quad\text{for }1\leq k\leq m
\]
and $
\norm{W}_{S^2_2}\leq-\alpha$. Consequently, the dual program of the preceding minimization problem is
\[
\sup\left\{
\alpha:
W\geq A_k\text{ for }1\leq k\leq m,\quad
\norm{W}_{S^2_2}\leq-\alpha
\right\},
\]
whose value is
\[
-\inf\left\{
\norm{W}_{S^2_2}:
W=W^*,\quad
W\geq A_k\text{ for }1\leq k\leq m
\right\}.
\]
Finally, the primal conic program \eqref{primal-conic} has an interior feasible point. Indeed, take
\[
P_k=\varepsilon\I_2,
\qquad
R=m\varepsilon\I_2,
\qquad
t=1,
\]
where $\epsi > 0$ is sufficiently small that $\norm{m\epsi\I_2}_{S^2_2}<1$. Then $P_k > 0$ for every $k$ and $
(R,1) \in \operatorname{int}\mathcal L$. Thus $x=(P_1,\ldots,P_m,R,t)$ belongs to $\operatorname{int}K$ and $\mathscr A x=b$. Since the optimal value is finite, \cite[Theorem~11.23 p.~295]{Gul10} gives strong duality. 
By \cite[Theorem~11.23 p.~295]{Gul10}, we conclude that
\begin{align*}
\sup_{\substack{P_k\geq0\\\sum_{k=1}^mP_k\neq0}}
\frac{\sum_{k=1}^m\tr(D_kQD_k^*P_k)}
{\norm{\sum_{k=1}^mP_k}_{S^2_2}}
&= \inf\bigl\{\|W\|_{S^2_2} : W=W^*, W\geq D_kQD_k^*\text{ for }1 \leq k \leq m\bigr\} \\
&\ov{\eqref{eq-fQ}}{=}f_B(Q).         
\end{align*}
Combining this identity with \eqref{eq-double-supremum} proves \eqref{eq-variational}.
\end{proof}

For the matrix $B$ in \eqref{eq-B}, we have
\begin{equation}
\label{D1-D2-D3}
D_1=\operatorname{diag}(1,2),\qquad
D_2=\operatorname{diag}(1,-2),\qquad
D_3=\operatorname{diag}(2,2\mathrm{i}).
\end{equation}
We record some elementary properties of the function $f_B$. First, it is positively homogeneous: for every positive matrix $Q\in\M_2$ and every scalar $t\geq0$, we have
\[
f_B(tQ)=tf_B(Q).
\]
Indeed, if $W\geq D_kQD_k^*$ for anyy $k$, then $tW \geq D_k(tQ)D_k^*$ for anyy $k$. This gives $f_B(tQ)\leq tf_B(Q)$. For $t>0$, the reverse inequality follows by applying the same argument with $t^{-1}$.

Recall that the order on $\M_2^{\mathrm{sa}}$ is defined by $Q_1 \leq Q_2 \Longleftrightarrow Q_2-Q_1 \geq 0$. The function $f_B$ is increasing for this order. Indeed, if $0 \leq Q_1 \leq Q_2$, then
\[
D_kQ_1D_k^*
\leq D_kQ_2D_k^*
\]
for any $k$. Hence every matrix $W$ which is admissible in the definition of $f_B(Q_2)$ is also admissible in the definition of $f_B(Q_1)$. Therefore
\[
f_B(Q_1)\leq f_B(Q_2).
\]
The function $f_B$ is also convex. Let $Q_1,Q_2 \geq 0$ and $0\leq t\leq1$. For $\varepsilon>0$, choose selfadjoint matrices $W_1,W_2 \in \M_2$ such that
\begin{equation}
\label{inter-Wj-4}
W_j\geq D_kQ_jD_k^*
\quad\text{for any }k
\quad\text{and}\quad
\norm{W_j}_{S^2_2}\leq f_B(Q_j)+\varepsilon,
\qquad j=1,2.
\end{equation}
Then
\[
tW_1+(1-t)W_2
\ov{\eqref{inter-Wj-4}}{\geq} tD_k Q_1 D_k^*+(1-t)D_kQ_2D_k^*
\geq D_k(tQ_1+(1-t)Q_2)D_k^*
\]
for any $k$. Thus
\begin{align*}
\MoveEqLeft
f_B(tQ_1+(1-t)Q_2)
\ov{\eqref{eq-fQ}}{\leq} \norm{tW_1+(1-t)W_2}_{S^2_2}\\
&\leq t\norm{W_1}_{S^2_2}+(1-t)\norm{W_2}_{S^2_2}
\ov{\eqref{inter-Wj-4}}{\leq} tf_B(Q_1)+(1-t)f_B(Q_2)+\epsi.         
\end{align*}
Letting $\epsi$ decrease to zero proves the convexity of $f_B$.

We next describe a symmetry of $f_B$. Let $U\in\M_2$ be a diagonal unitary. Since $U$ commutes with every $D_k$, we have
\[
D_k(UQU^*)D_k^*
=U(D_kQD_k^*)U^*.
\]
If $W \geq D_kQD_k^*$ for any $k$, then
\[
UWU^*\geq D_k(UQU^*)D_k^*
\]
for any $k$. Moreover, $\norm{UWU^*}_{S^2_2}=\norm{W}_{S^2_2}$. It follows from \eqref{eq-fQ} that $
f_B(UQU^*) \leq f_B(Q)$. Applying the same argument to $U^*$ gives the reverse inequality. Hence
\begin{equation}
\label{eq-diagonal-unitary-invariance}
f_B(UQU^*)
=f_B(Q).
\end{equation}
We finally parametrize, up to this invariance, all the matrices which occur in the supremum in \eqref{eq-variational}. Let
\[
Q
\ov{\mathrm{def}}{=}\begin{bmatrix}a&\zeta\\\overline{\zeta}&d\end{bmatrix}\geq0
\quad \text{with} \quad
\norm{Q}_{S^2_2}=1.
\]
Write $\zeta=x\e^{\mathrm{i}\theta}$ with $x \geq 0$ and $\theta \in [0,2\pi]$. Conjugating $Q$ by the diagonal unitary $\operatorname{diag}(\e^{-\mathrm{i}\theta},1)$ makes its off-diagonal entry equal to $x$. In view of \eqref{eq-diagonal-unitary-invariance}, this conjugation does not change the value of $f_B(Q)$. We may therefore suppose that
\[
Q
=\begin{bmatrix}
a&x\\x&d
\end{bmatrix},
\quad \text{with} \quad x \geq 0.
\]
Set $
q_0\ov{\mathrm{def}}{=}\frac{a+d}{2}$
and $z\ov{\mathrm{def}}{=}\frac{a-d}{2}$. Then $q_0+z=a$ and $q_0-z=d$. So
\begin{equation}
\label{def-Q-x-z}
Q=Q(x,z)
=\begin{bmatrix}
q_0+z&x\\x&q_0-z
\end{bmatrix}.
\end{equation}
We have $Q^2 
\ov{\eqref{def-Q-x-z}}{=} \begin{bmatrix}
  (q_0+z)^2+x^2   &  2q_0x \\
  2q_0x   &  (q_0-z)^2+x^2  \\
\end{bmatrix}$. The condition $\norm{Q}_{S^2_2}=1$ is equivalent to
\begin{equation}
\label{inter-29029}
1
=\tr(Q^2)
=(q_0+z)^2+(q_0-z)^2+2x^2
=2(q_0^2+x^2+z^2).
\end{equation}
Since $Q \geq 0$, we have $q_0\geq0$, and consequently
\begin{equation}
\label{eq-q0-formula}
q_0
\ov{\eqref{inter-29029}}{=} \sqrt{\frac12-x^2-z^2}.
\end{equation}
The eigenvalues of $Q(x,z)$ are $
q_0-\sqrt{x^2+z^2}$
and $
q_0+\sqrt{x^2+z^2}$. Thus the matrix $Q(x,z)$ is positive if and only if $q_0\geq\sqrt{x^2+z^2}$. Using \eqref{eq-q0-formula}, this condition is equivalent to $x^2+z^2\leq\frac14$. We have therefore proved that, modulo conjugation by diagonal unitaries, the positive matrices $Q\in\M_2$ satisfying $\norm{Q}_{S^2_2}=1$ are exactly the matrices
\begin{equation}
\label{eq-Qxz}
Q(x,z)
\ov{\eqref{def-Q-x-z}}{=}
\begin{bmatrix}
q_0+z&x\\
x&q_0-z
\end{bmatrix},
\qquad
q_0
=\sqrt{\frac12-x^2-z^2},
\end{equation}
where
\begin{equation}
\label{eq-half-disk}
x\geq0,\qquad x^2+z^2\leq\frac14.
\end{equation}

\section{An analytic computation of the scalar norm}
\label{sec-analytic-scalar-norm}

Now, we specialize the variational formula of the preceding section to the symbol $B$ in \eqref{eq-B}. Its symmetries first reduce the optimization over positive $2\times2$ matrices to a two-parameter family. We then solve the associated problem analytically and reduce the remaining maximization to a one-variable algebraic problem. This yields the exact ordinary norm of $M_B$ and, in particular, a strictly stronger estimate than the rational upper bound required in Theorem~\ref{thm-main}. The argument is entirely analytic and independent of the certified computation given in the following section.

Recall that $
B=
\begin{bmatrix}
1&2\\
1&-2\\
2&2\mathrm{i}
\end{bmatrix}$, $
D_1 \ov{\eqref{D1-D2-D3}}{=}\operatorname{diag}(1,2)$, $
D_2 \ov{\eqref{D1-D2-D3}}{=} \operatorname{diag}(1,-2)$,
$D_3 \ov{\eqref{D1-D2-D3}}{=}\operatorname{diag}(2,2\mathrm{i})$. For a positive matrix $Q\in\M_2$, recall that
\[
f_B(Q)
\ov{\eqref{eq-fQ}}{=}
\inf\left\{
\norm{W}_{S^2_2}:
W=W^*,
\quad
W\geq D_kQD_k^*
\text{ for }1\leq k\leq3
\right\}.
\]
Define
\[
h(c)
\ov{\mathrm{def}}{=}69c^3-568c^2-1168c+1344.
\]
The polynomial $h$ has a unique zero $c_*$ in the interval $
\frac{39}{4}<c_*<\frac{781}{80}$. Set
\begin{equation}
\label{eq-Lambda-star}
\Lambda_*
\ov{\mathrm{def}}{=}
\frac{16c_*^2(3c_*-7)}
{23c_*^2-36c_*-144}.
\end{equation}

\begin{thm}
\label{thm-exact-scalar-norm}
We have
\begin{equation}
\label{eq-exact-scalar-norm}
\norm{M_B}_{S^4_{3,2}\to S^4_{3,2}}^4
=
\Lambda_*.
\end{equation}
Moreover, we have
\begin{equation}
\label{eq-analytic-rational-upper}
\Lambda_*
<
\frac{20033}{1000}
<
\left(\frac{1119}{250}\right)^2.
\end{equation}
Consequently, we have
\[
\norm{M_B}_{S^4_{3,2}\to S^4_{3,2}}^2
<
\frac{1119}{250}.
\]
\end{thm}

We start by solving analytically the two-dimensional semidefinite problem defining $f_B$. We shall also use the classical Lagrange multiplier theorem. Recall that if $F,G\co\R^n\to\R$ are continuously differentiable, $x_0$ is a local extremum of $F$ under the constraint $G(x)=0$, and $\nabla G(x_0)\neq0$, then there exists $\mu\in\R$ such that
\[
\nabla F(x_0)=\mu\nabla G(x_0).
\]
See, for example, \cite[Section~12.2, Theorem~12.1, pp.~327--329]{NoW06}, see also \cite[Section~9.3, pp.~242--247]{Lue97}. 

\begin{lemma}
\label{lem-one-parameter-SDP}
Consider a matrix $
Q
\ov{\mathrm{def}}{=}
\begin{bmatrix}
a&r\\
r&b
\end{bmatrix}$ such that $\norm{Q}_{S^2_2}=1$ with $r\geq0$. If $r=0$, then $f_B(Q)=4$. Suppose that $r>0$ and set $s\ov{\mathrm{def}}{=}\frac{a}{b}$ and  $\kappa\ov{\mathrm{def}}{=}\frac{r^2}{ab}$. Then $s > 0$, $0 < \kappa \leq 1$, and
\begin{equation}
\label{eq-f-exact-Phi}
f_B(Q)^2
=
\min_{\frac32<\tau\leq4}
\Phi(s,\kappa,\tau),
\end{equation}
where
\begin{equation}
\label{eq-Phi-definition}
\Phi(s,\kappa,\tau)
\ov{\mathrm{def}}{=}
\frac{
s^2U(\tau)^2+
\bigl(4+\kappa R(\tau)\bigr)^2+
2\kappa s\tau^2
}{
s^2+1+2\kappa s
},
\end{equation}
with
\begin{equation}
\label{eq-U-R-definition}
U(\tau)
\ov{\mathrm{def}}{=}
\frac{\tau(3\tau+8)}{4(2\tau-3)},
\qquad
R(\tau)
\ov{\mathrm{def}}{=}
\frac{4(2\tau-3)}{3}.
\end{equation}
If $\kappa>0$, the minimum in \eqref{eq-f-exact-Phi} is attained at a unique point of the interval $(\frac{3}{2},4)$.
\end{lemma}

\begin{proof}
Let
\[
A_k\ov{\mathrm{def}}{=} D_kQD_k^*,
\qquad
1\leq k\leq3.
\]
Explicitly, we have
\[
A_1=
\begin{bmatrix}
a&2r\\
2r&4b
\end{bmatrix},
\qquad
A_2=
\begin{bmatrix}
a&-2r\\
-2r&4b
\end{bmatrix},
\qquad
A_3=
\begin{bmatrix}
4a&-4\mathrm{i}r\\
4\mathrm{i}r&4b
\end{bmatrix}.
\]
Put $J\ov{\mathrm{def}}{=}\operatorname{diag}(1,-1)$ and define $
\Theta(W)
\ov{\mathrm{def}}{=}J\overline{W}J$. The map $\Theta$ is an isometry of $\M_2^{\mathrm{sa}}$ for the Hilbert--Schmidt norm, and
\[
\Theta(A_1)=A_2,
\qquad
\Theta(A_2)=A_1,
\qquad
\Theta(A_3)=A_3.
\]
Hence the feasible set in the definition of $f_B(Q)$ is invariant under $\Theta$. Averaging a minimizer with its image under $\Theta$, we see that a minimizer may be chosen in the form
\[
W=
\begin{bmatrix}
u&\mathrm{i}\beta\\
-\mathrm{i}\beta&v
\end{bmatrix},
\qquad
u,v,\beta\in\R.
\]
The inequalities $W\geq A_k$ are equivalent to
\[
u\geq4a,
\qquad
v\geq4b, 
(u-a)(v-4b)\geq4r^2+\beta^2,
\]
and
\[
(u-4a)(v-4b)\geq(\beta+4r)^2.
\]
Set $p\ov{\mathrm{def}}{=}u-4a$, $\sigma\ov{\mathrm{def}}{=}v-4b$. Thus $p,\sigma\geq0$, and the problem is to minimize
\begin{equation}
\label{eq-objective-p-sigma-beta}
(p+4a)^2+(\sigma+4b)^2+2\beta^2
\end{equation}
under
\begin{equation}
\label{eq-two-determinant-constraints}
(p+3a)\sigma\geq\beta^2+4r^2,
\qquad
p\sigma\geq(\beta+4r)^2.
\end{equation}
Replacing $\beta$ by $0$ when $\beta>0$, or by $-4r$ when $\beta<-4r$, preserves feasibility and decreases the objective. A direct perturbation excludes the endpoint $\beta=-4r$ at a minimizer when $r>0$. We may therefore suppose that
\[
-4r<\beta\leq0.
\]
In particular, the second inequality in \eqref{eq-two-determinant-constraints} implies $p,\sigma>0$.

We claim that both inequalities in \eqref{eq-two-determinant-constraints} are equalities at a minimizer. At least one of them must be an equality, since otherwise $p$ could be slightly decreased. Suppose first that the second inequality is strict. Then the first one is an equality. If $\beta\neq0$, moving $\beta$ slightly towards $0$ decreases the first right-hand side and the objective, while the strict second inequality remains valid. Thus $\beta=0$. We would then have
\[
(p+3a)\sigma=4r^2
\]
and hence $
p\sigma<4r^2$, whereas the second inequality requires $
p\sigma>16r^2$. This is impossible. Suppose now that the first inequality is strict. Then
\[
p\sigma=(\beta+4r)^2.
\]
Locally, the first constraint is irrelevant, and the Lagrange multiplier theorem applied to the last equality gives a number $\mu>0$ such that
\[
2(p+4a)=\mu\sigma,
\qquad
2(\sigma+4b)=\mu p,
\]
and $
2\beta+\mu(\beta+4r)=0$. Multiplying the first two identities gives $\mu>2$. Solving them yields
\[
p=\frac{8(2a+b\mu)}{\mu^2-4},
\qquad
\sigma=\frac{8(a\mu+2b)}{\mu^2-4},
\]
whereas the third identity gives
\[
\beta+4r=\frac{8r}{\mu+2}.
\]
The equality $p\sigma=(\beta+4r)^2$ therefore implies
\[
(2a+b\mu)(a\mu+2b)
=
r^2(\mu-2)^2.
\]
Since $Q\geq0$, we have $r^2\leq ab$. However,
\[
(2a+b\mu)(a\mu+2b)-ab(\mu-2)^2
=
2\mu(a+b)^2
>
0,
\]
which is a contradiction. Thus both inequalities in \eqref{eq-two-determinant-constraints} are equalities. Subtracting them gives
\[
3a\sigma
=
\beta^2+4r^2-(\beta+4r)^2
=
-4r(2\beta+3r).
\]
Consequently, $\beta<-\frac{3r}{2}$. Write
\[
\beta=-\tau r,
\qquad
\frac32<\tau<4.
\]
The two equalities in \eqref{eq-two-determinant-constraints} give
\[
\sigma
=
\frac{4r^2(2\tau-3)}{3a}
\quad \text{and} \quad
p
=
\frac{3a(4-\tau)^2}{4(2\tau-3)}.
\]
Hence
\[
u=p+4a
=
a\frac{\tau(3\tau+8)}{4(2\tau-3)}
\ov{\eqref{eq-U-R-definition}}{=}
aU(\tau)
\quad \text{and} \quad
v=\sigma+4b
=
4b+\frac{r^2}{a}R(\tau).
\]
Conversely, these formulas define an admissible matrix $W$ for any $\frac{3}{2}<\tau\leq4$. Therefore
\[
f_B(Q)^2
=
\min_{\frac32<\tau\leq4}
\left(
a^2U(\tau)^2+
\left(4b+\frac{r^2}{a}R(\tau)\right)^2+
2\tau^2r^2
\right).
\]
Since $a=sb$, $r^2=\kappa ab=\kappa sb^2$ and
\[
1=a^2+b^2+2r^2
=b^2(s^2+1+2\kappa s),
\]
we obtain \eqref{eq-f-exact-Phi}. It remains to prove uniqueness of the minimizer. Put $y \ov{\mathrm{def}}{=}2\tau-3$. The numerator of $\Phi$ becomes
\[
N_{s,\kappa}(y)
=
s^2
\left(
\frac{(y+3)(3y+25)}{16y}
\right)^2
+
\left(4+\frac{4\kappa y}{3}\right)^2
+
\frac{\kappa s}{2}(y+3)^2.
\]
A direct differentiation gives
\[
N_{s,\kappa}''(y)
=
\frac{
(64\kappa+9s)^2y^4+
45900s^2y+
151875s^2
}{
1152y^4
}
>
0.
\]
Thus $N_{s,\kappa}$ is strictly convex on $(0,\infty)$. Moreover, we have
\[
N_{s,\kappa}'(5)
=
\frac{8\kappa(20\kappa+9s+12)}{9}.
\]
If $\kappa>0$, the unique minimizer belongs to $(0,5)$, which corresponds to $\tau\in(3/2,4)$. If $\kappa=0$, the unique minimizer is $y=5$, or equivalently $\tau=4$.
\end{proof}

Define
\begin{equation}
\label{eq-G-definition}
G(s,\kappa)
\ov{\mathrm{def}}{=} \min_{\frac32 < \tau \leq 4} \Phi(s,\kappa,\tau),
\quad
s>0,
\quad
0\leq\kappa\leq1.
\end{equation}

\begin{lemma}
\label{lem-boundary-G}
We have $
G(s,0)=16$ for any $s>0$, and
\[
\lim_{s\to0}G(s,\kappa)
=
\lim_{s\to\infty}G(s,\kappa)
=
16
\]
uniformly for $0\leq\kappa\leq1$. Moreover, we have
\begin{equation}
\label{eq-rank-one-boundary}
G(s,1)<20
\end{equation}
for every $s>0$.
\end{lemma}

\begin{proof}
First note that
\[
U(\tau)-4
\ov{\eqref{eq-U-R-definition}}{=}\frac{\tau(3\tau+8)}{4(2\tau-3)}-4
=\frac{3(\tau-4)^2}{4(2\tau-3)}
\geq 0.
\]
Thus
\[
G(s,0)
\ov{\eqref{eq-G-definition}}{=} \min_{\frac32 < \tau \leq 4} \Phi(s,0,\tau)
\ov{\eqref{eq-Phi-definition}}{=} \min_{\frac32 < \tau \leq 4} \frac{s^2U(\tau)^2+16}{s^2+1}
= 16.
\]
For any positive matrix $Q$ with $\norm{Q}_{S^2_2}=1$, we have
\[
D_3QD_3^*
=
4VQV^*,
\qquad
V=\operatorname{diag}(1,\i).
\]
Hence every admissible $W$ satisfies $W \geq 4VQV^*$. The Hilbert--Schmidt norm is increasing on the positive cone, so $
f_B(Q)
\geq 4$. Consequently, we have $
G(s,\kappa)
\geq 16$. For $s \to \infty$, taking $\tau=4$ gives
\[
G(s,\kappa)
\leq
\frac{
16s^2+
\left(4+\frac{20\kappa}{3}\right)^2+
32\kappa s
}{
s^2+1+2\kappa s
},
\]
whose right-hand side tends uniformly to $16$. For $s\to0$, take $
2\tau-3=\sqrt{s}$. Substitution in \eqref{eq-Phi-definition} shows uniformly for $0\leq\kappa\leq1$ that
\[
\Phi(s,\kappa,\tau)
=16+o(1).
\]
Together with the lower bound $G\geq16$, this proves the two limits. It remains to consider $\kappa=1$. Set
\[
\tau_s
\ov{\mathrm{def}}{=}\frac{4s+3}{s+2}.
\]
Then $\frac{3}{2}<\tau_s<4$, and direct simplification gives
\begin{equation}
\label{inter-Phi-Phi}
\Phi(s,1,\tau_s)
\ov{\eqref{eq-Phi-definition}}{=}
\frac{
(4s+3)^2(144s^2+648s+1249)
}{
144(s+1)^2(s+2)^2
}.
\end{equation}
Therefore
\[
20-\Phi(s,1,\tau_s)
\ov{\eqref{inter-Phi-Phi}}{=}
\frac{
576s^4+3456s^3+608s^2-1248s+279
}{
144(s+1)^2(s+2)^2
}.
\]
The numerator admits the decomposition
\[
\begin{split}
576s^4+3456s^3+608s^2-1248s+279
={}&
\frac{1824s^2-32s+5}{3}
+\frac{64}{3}(3s-1)^2(3s^2+20s+13).
\end{split}
\]
The discriminant of $1824s^2-32s+5$ is
\[
(-32)^2-4\cdot1824\cdot5
=
-35456<0.
\]
Thus the numerator is strictly positive for every $s\geq0$. Hence $
G(s,1)
\leq
\Phi(s,1,\tau_s)
<
20$.
\end{proof}

\begin{lemma}
\label{lem-scalar-witness-over-twenty}
We have $
\norm{M_B}_{S^4_{3,2}\to S^4_{3,2}}^4
>20$.
\end{lemma}

\begin{proof}
Set $
X_0\ov{\mathrm{def}}{=}
\begin{bmatrix}
5&8\\
5&8\mathrm{i}\\
10&-8-8\mathrm{i}
\end{bmatrix}$. Then $
B\circ X_0
=
\begin{bmatrix}
5&16\\
5&-16\mathrm{i}\\
20&16-16\mathrm{i}
\end{bmatrix}$. A direct computation gives
\begin{equation}
\label{un-produit}
X_0^*X_0
=
\begin{bmatrix}
150&-40-40\mathrm{i}\\
-40+40\mathrm{i}&256
\end{bmatrix}
\end{equation}
and
\begin{equation}
\label{un-autre-produit}
(B\circ X_0)^*(B\circ X_0)
=
\begin{bmatrix}
450&400-400\mathrm{i}\\
400+400\mathrm{i}&1024
\end{bmatrix}.
\end{equation}
Consequently, we have
\begin{equation}
\label{inter-4776-B}
\norm{X_0}_{S^4_{3,2}}^4
=
\tr (X_0^*X_0)^2
\ov{\eqref{un-produit}}{=}\tr \begin{bmatrix}
25700&-16240-16240\mathrm{i}\\
-16240+16240\mathrm{i}&68736
\end{bmatrix}
=94436
\end{equation}
and
\begin{equation}
\label{inter-4776-C}
\norm{B\circ X_0}_{S^4_{3,2}}^4
=\tr
\big(
(B\circ X_0)^*(B\circ X_0)
\big)^2
\ov{\eqref{un-autre-produit}}{=} \tr \begin{bmatrix}
522500&589600-589600\mathrm{i}\\
589600+589600\mathrm{i}&1368576
\end{bmatrix}
=1891076.
\end{equation}
Therefore
\[
\norm{M_B}_{S^4_{3,2}\to S^4_{3,2}}^4
\ov{\eqref{inter-4776-B}\eqref{inter-4776-C}}{\geq}
\frac{1891076}{94436}
= \frac{472769}{23609}
= 20+\frac{589}{23609}
> 20.
\]
\end{proof}

\begin{proof}[of Theorem~\ref{thm-exact-scalar-norm}]
By Proposition~\ref{prop-variational} and Lemma~\ref{lem-one-parameter-SDP}, we have
\begin{equation}
\label{eq-norm-fourth-sup-G}
\norm{M_B}_{S^4_{3,2}\to S^4_{3,2}}^4
\ov{\eqref{eq-variational}}{=} \sup_{\substack{Q \geq 0 \\ \|Q\|_{S^2_2}=1} }f_B(Q)
=\sup_{\substack{s>0\\0\leq\kappa\leq1}}
G(s,\kappa).
\end{equation}
The strict convexity established in Lemma~\ref{lem-one-parameter-SDP} implies that $G$ is continuous. Lemma~\ref{lem-boundary-G} shows that $G$ extends continuously to the compactification obtained by adjoining $s=0$ and $s=\infty$.

By Lemma~\ref{lem-scalar-witness-over-twenty}, the supremum in \eqref{eq-norm-fourth-sup-G} is strictly greater than $20$. Lemma~\ref{lem-boundary-G} shows that every boundary value is at most $20$. Thus the maximum is attained at a point satisfying
\[
s>0,
\qquad
0<\kappa<1.
\]
Let $\tau\in(3/2,4)$ be the unique minimizer in \eqref{eq-G-definition}, and put $L \ov{\mathrm{def}}{=} G(s,\kappa)>20$. Set
\[
A \ov{\mathrm{def}}{=} U(\tau)^2,
\qquad
C \ov{\mathrm{def}}{=} R(\tau),
\qquad
T \ov{\mathrm{def}}{=} \tau^2.
\]
The stationarity conditions with respect to $s$ and $\kappa$ give
\begin{equation}
\label{eq-stationarity-s}
(A-L)s
=(L-T)\kappa
\quad \text{and} \quad
(L-T)s
=C(4+\kappa C).
\end{equation}
Since
\[
U(\tau)-\tau
\ov{\eqref{eq-U-R-definition}}{=}\frac{\tau(3\tau+8)}{4(2\tau-3)}-\tau
= \frac{5\tau(4-\tau)}{4(2\tau-3)}
> 0,
\]
we have $A>T$. Equations \eqref{eq-stationarity-s} imply $
A>L>T$. Put $
\alpha \ov{\mathrm{def}}{=}A-L$ and $\delta\ov{\mathrm{def}}{=}L-T$. Then $
\alpha s=\delta\kappa$ and $\delta s=C(4+\kappa C)$. Using the identity defining $L$, we obtain $
L
=
(4+\kappa C)^2-\delta\kappa s$. Consequently, we have
\[
L
=
s^2
\left(
\frac{\delta^2}{C^2}-\alpha
\right).
\]
On the other hand, eliminating $\kappa$ between the two stationarity equations gives
\[
s=
\frac{4C\delta}{\delta^2-\alpha C^2}.
\]
Combining the last two identities yields
\begin{equation}
\label{eq-critical-basic}
(L-16)(L-T)^2
=
L(A-L)C^2.
\end{equation}
Substituting the definitions of $A$, $C$ and $T$ in \eqref{eq-critical-basic} and clearing denominators gives
\begin{equation}
\label{eq-critical-polynomial-one}
9L^3
+
(46\tau^2-192\tau)L^2
+
(224\tau^2-48\tau^3)L
-
144\tau^4
=
0.
\end{equation}
The stationarity condition with respect to $\tau$ is
\[
A'
+
2(A-L)\frac{C'}{C}
+
2(A-L)\frac{T'}{L-T}
=
0.
\]
Substitution of $U$, $R$ and their derivatives gives
\begin{equation}
\label{eq-critical-polynomial-two}
-32L^2\tau
+
48L^2
-
23L\tau^3
+
180L\tau^2
-
112L\tau
+
12\tau^4
+
32\tau^3
=
0.
\end{equation}
Set $
c\ov{\mathrm{def}}{=} \frac{L}{\tau}$. Since $L>20$ and $\tau<4$, we have $c>5$. Dividing \eqref{eq-critical-polynomial-one} by $\tau^3$ gives
\begin{equation}
\label{eq-critical-c-one}
9c^3
+
(46c^2-48c-144)\tau
-
192c^2
+
224c
=
0.
\end{equation}
Similarly, \eqref{eq-critical-polynomial-two} gives
\begin{equation}
\label{eq-critical-c-two}
(23c-12)\tau^2
+
(32c^2-180c-32)\tau
-
48c^2
+
112c
=
0.
\end{equation}
Equation \eqref{eq-critical-c-one} yields
\begin{equation}
\label{eq-tau-from-c-first}
\tau
=
-\frac{c(9c^2-192c+224)}
{2(23c^2-24c-72)}.
\end{equation}
Substituting \eqref{eq-tau-from-c-first} into \eqref{eq-critical-c-two} gives
\[
0
=
-\frac{
15c(c-4)^2(11c-4)
h(c)
}{
4(23c^2-24c-72)^2
}.
\]
Since $c>5$, all the factors except $h(c)$ are nonzero. Therefore $h(c)=0$. Now, we verify that $h$ has a unique zero larger than $5$. We have
\[
h''(c)=414c-1136>0
\]
for $c\geq5$. Hence $h'$ is strictly increasing on $[5,\infty)$. Moreover, we have
\[
h'(5)=-1673<0,
\qquad
h'(8)=2992>0,
\]
and
\[
h(5)=-10071<0,
\qquad
h(8)=-9024<0.
\]
It follows that $h$ has a unique zero in $(8,\infty)$. Finally, we obtain
\[
h\left(\frac{39}{4}\right)
=
-\frac{5517}{64}<0
\quad \text{and} \quad
h\left(\frac{781}{80}\right)
=
\frac{3557289}{512000}>0.
\]
Thus this zero is precisely $c_*$. The identity
\begin{align*}
-\frac{c(9c^2-192c+224)}
{2(23c^2-24c-72)}
-
\frac{16c(3c-7)}
{23c^2-36c-144}
=
-\frac{
3c(c-4)h(c)
}{
2(23c^2-36c-144)(23c^2-24c-72)
}
\end{align*}
shows that, when $h(c)=0$,
\begin{equation}
\label{inter-fin-fin}
\tau
=
\frac{16c(3c-7)}
{23c^2-36c-144}.
\end{equation}
Hence
\[
L=c\tau
\ov{\eqref{inter-fin-fin}}{=}
\frac{16c^2(3c-7)}
{23c^2-36c-144}.
\]
Since $c=c_*$, this proves $
L=\Lambda_*$. As $L$ is the global maximum in \eqref{eq-norm-fourth-sup-G}, we obtain \eqref{eq-exact-scalar-norm}.
 It remains to establish the rational upper estimate. Define
\[
\Lambda(c)
\ov{\mathrm{def}}{=}
\frac{16c^2(3c-7)}
{23c^2-36c-144}.
\]
A direct differentiation gives
\[
\Lambda'(c)
=
\frac{
48c(23c^3-72c^2-348c+672)
}{
(23c^2-36c-144)^2
}.
\]
The cubic in the numerator is positive for $c\geq9$. Indeed, its value at $9$ is $8475$, and its derivative is positive on $[9,\infty)$. Thus $\Lambda$ is increasing on $
\left[\frac{39}{4},\frac{781}{80}\right]$. Consequently, we have
\[
\Lambda_*
<
\Lambda\left(\frac{781}{80}\right)
=
\frac{1087560463}{54291115}.
\]
Moreover,
\[
\frac{20033}{1000}
-
\frac{1087560463}{54291115}
=
\frac{10688759}{10858223000}
>
0,
\]
and
\[
\left(\frac{1119}{250}\right)^2
-
\frac{20033}{1000}
=
\frac{197}{125000}
>
0.
\]
This proves \eqref{eq-analytic-rational-upper}.
\end{proof}

\begin{remark} \normalfont
Numerically, we have
\[
c_*=9.761569762128026\ldots,
\qquad
\Lambda_*=20.030226756833782\ldots.
\]
Thus
\[
\norm{M_B}_{S^4_{3,2}\to S^4_{3,2}}
=
2.115541097034116\ldots.
\]
The numerical value obtained here agrees with the non-certified BFGS approximation, but the proof above is entirely analytic.
\end{remark}

\section{An independent certified proof of the three-by-three separation}
\label{sec-certificate}

The preceding section determined the exact scalar norm by a fully analytic argument. In this section, we give an independent computer-assisted proof of the slightly weaker rational upper bound appearing in Theorem~\ref{thm-main}. We retain this second proof both as an independent verification and because it has a different, convex-geometric nature: the continuous optimization problem is reduced to finitely many inequalities between $2\times2$ matrices, all of which are verified in exact rational arithmetic. Combined with the explicit lower witness for the second amplification, this provides a second proof of the strict separation in Theorem~\ref{thm-main}.

\paragraph{Strategy of the proof.}
By Proposition~\ref{prop-variational}, it is enough to prove that
\[
f_B(Q(x,z))\leq\frac{1119}{250}
\]
for every $(x,z)$ in the half-disk
\[
\mathcal{D}
\ov{\mathrm{def}}{=}
\left\{(x,z)\in\R^2:x\geq0\text{ and }x^2+z^2\leq\frac14\right\}.
\]
By the definition of $f_B$, this amounts to finding, for every $(x,z)\in\mathcal{D}$, a selfadjoint matrix $W$ such that
\[
W\geq D_kQ(x,z)D_k^*
\quad\text{for }1\leq k\leq3
\quad\text{and}\quad
\norm{W}_{S^2_2}\leq\frac{1119}{250}.
\]
It is of course impossible to check the uncountably many points of $\mathcal{D}$ separately. The idea is to divide $\mathcal{D}$ into finitely many rectangles and to construct one certificate for each rectangle.

The difficulty is that the map $(x,z)\mapsto Q(x,z)$ is not affine because its diagonal entries contain the function
\[
q_0(x,z)
\ov{\mathrm{def}}{=} \sqrt{\frac12-x^2-z^2}.
\]
On each rectangle $R$, we majorize $q_0$ by an affine function. This gives an affine matrix-valued function $\widehat Q_R$ satisfying
\[
Q(x,z)\leq\widehat Q_R(x,z)
\]
for $(x,z)\in R\cap\mathcal{D}$. Since $f_B$ is increasing, we have
\[
f_B(Q(x,z))\leq f_B(\widehat Q_R(x,z)).
\]
Moreover, $\widehat Q_R$ is affine and $f_B$ is convex. It is therefore sufficient to control $f_B(\widehat Q_R)$ at the four vertices of $R$. The original continuous problem is thus reduced to finitely many inequalities involving $2\times2$ matrices with rational entries. These inequalities can be checked exactly by a computer.

\begin{figure}[ht]
\centering
\begin{tikzpicture}[scale=8]
\fill[blue!8] (0,-0.5) arc[start angle=-90,end angle=90,radius=0.5] -- cycle;
\draw[blue,thick] (0,-0.5) arc[start angle=-90,end angle=90,radius=0.5];
\draw[thick] (0,-0.5) rectangle (0.5,0.5);
\draw[dashed] (0.25,-0.5)--(0.25,0.5);
\draw[dashed] (0,-0.25)--(0.5,-0.25);
\draw[dashed] (0,0)--(0.5,0);
\draw[dashed] (0,0.25)--(0.5,0.25);
\draw[dashed] (0.125,0)--(0.125,0.25);
\draw[dashed] (0.375,0)--(0.375,0.25);
\fill[orange!35] (0.25,0)--(0.375,0.125);
\draw[orange!80!black,thick] (0.25,0) rectangle (0.375,0.125);
\draw[->] (-0.03,0)--(0.57,0) node[right] {$x$};
\draw[->] (0,-0.57)--(0,0.57) node[above] {$z$};
\node[blue] at (0.17,-0.16) {$\mathcal{D}$};
\node[orange!80!black] at (0.315,0.17) {$R$};
\node[below left] at (0,-0.5) {$-\frac12$};
\node[above left] at (0,0.5) {$\frac12$};
\node[below] at (0.5,0) {$\frac12$};
\end{tikzpicture}
\caption{The half-disk $\mathcal{D}$ inside the initial rectangle. Each rectangle which is neither discarded nor certified is divided into four dyadic subrectangles.}
\label{fig-dyadic-cover}
\end{figure}
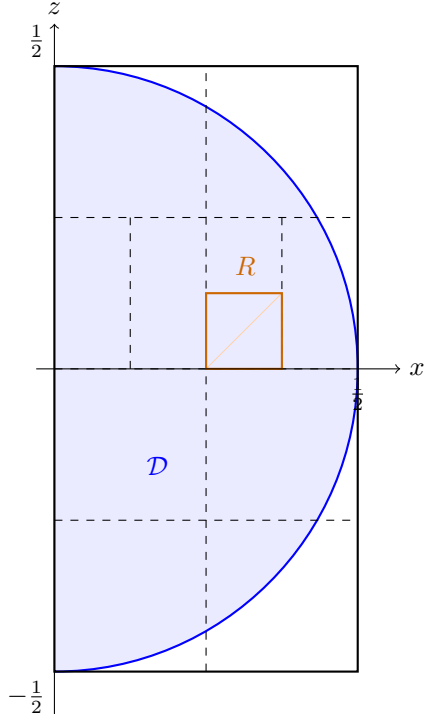

\begin{prop}
\label{prop-certified-upper}
For the matrix $B$ defined in \eqref{eq-B}, we have
\[
\norm{M_B}_{S^4_{3,2} \to S^4_{3,2}}^2
\leq \frac{1119}{250}.
\]
\end{prop}

\begin{proof}
Set
\[
T\ov{\mathrm{def}}{=}\frac{1119}{250}
\quad\text{and}\quad
\mathcal{D}
\ov{\mathrm{def}}{=}
\left\{(x,z)\in\R^2:x\geq0\text{ and }x^2+z^2\leq\frac14\right\}.
\]
According to Proposition~\ref{prop-variational}, we have to prove that $
f_B(Q(x,z))\leq T$ for any $(x,z) \in \mathcal{D}$. We first explain how a single rectangle is certified. Let
\[
R=[a_0,a_1]\times[c_0,c_1]
\subseteq
\left[0,\frac12\right]\times
\left[-\frac12,\frac12\right]
\]
be a rectangle with dyadic rational endpoints. The verifier constructs rational numbers $e>0$, $g_1$ and $g_2$ such that the affine function
\[
\ell_R(x,z)
\ov{\mathrm{def}}{=}
e+g_1x+g_2z
\]
satisfies
\begin{equation}
\label{eq-affine-majorant-condition}
2e^2\geq1+g_1^2+g_2^2.
\end{equation}
We claim that
\begin{equation}
\label{eq-affine-majorant-q0}
\sqrt{\frac12-x^2-z^2}\leq\ell_R(x,z)
\end{equation}
whenever $x^2+z^2\leq\frac12$. Indeed, put $
q_0\ov{\mathrm{def}}{=}\sqrt{\frac12-x^2-z^2}$. Then $
q_0^2+x^2+z^2=\frac12$. The Euclidean Cauchy--Schwarz inequality gives
\begin{align*}
\MoveEqLeft
q_0-g_1x-g_2z
=\left\langle(q_0,x,z),(1,-g_1,-g_2)\right\rangle\\
&\leq\sqrt{q_0^2+x^2+z^2}\sqrt{1+g_1^2+g_2^2}
=\frac{1}{\sqrt{2}}\sqrt{1+g_1^2+g_2^2}
\ov{\eqref{eq-affine-majorant-condition}}{\leq}e.         
\end{align*}
This is exactly \eqref{eq-affine-majorant-q0}. The construction of $\ell_R$ can be understood geometrically. The function $
(x,z) \mapsto\sqrt{\frac12-x^2-z^2}$ is concave. Its tangent plane at a point $(x_0,z_0)$ is an affine majorant. If
\[
q_0^0
\ov{\mathrm{def}}{=}\sqrt{\frac12-x_0^2-z_0^2},
\]
then this tangent plane has the form $e+g_1x+g_2z$, where
\[
g_1=-\frac{x_0}{q_0^0},
\qquad
g_2=-\frac{z_0}{q_0^0}
\quad \text{and} \quad
e=\frac{1}{2q_0^0}.
\]
For these values, equality holds in \eqref{eq-affine-majorant-condition}. The program starts from such a tangent plane near the centre of $R$, replaces its coefficients by nearby rational numbers and increases $e$, if necessary, until \eqref{eq-affine-majorant-condition} is verified exactly.

Some vertices of $R$ may lie outside $\mathcal{D}$. The verifier therefore chooses an additional nonnegative rational number $s_R$ and sets
\[
\widehat q_{0,R}(x,z)
\ov{\mathrm{def}}{=}
\ell_R(x,z)+s_R
\]
and
\begin{equation}
\label{eq-Qhat}
\widehat Q_R(x,z)
\ov{\mathrm{def}}{=}
\begin{bmatrix}
\widehat q_{0,R}(x,z)+z&x\\
x&\widehat q_{0,R}(x,z)-z
\end{bmatrix}.
\end{equation}
The number $s_R$ is chosen so that $\widehat Q_R$ is positive at each of the four vertices of $R$. Since $\widehat Q_R$ is affine and the cone of positive matrices is convex, it follows that
\[
\widehat Q_R(x,z)\geq0
\]
throughout $R$. Moreover, for any $(x,z) \in R \cap \mathcal{D}$, \eqref{eq-affine-majorant-q0} gives
\[
\widehat q_{0,R}(x,z)
\geq\ell_R(x,z)
\geq q_0(x,z).
\]
Consequently, we have
\begin{equation}
\label{eq-Qhat-majorizes-Q}
\widehat Q_R(x,z)-Q(x,z)
=
\big(\widehat q_{0,R}(x,z)-q_0(x,z)\big)\I_2
\geq0.
\end{equation}
Since $f_B$ is increasing for the Loewner order, we deduce that
\begin{equation}
\label{eq-fQ-fQhat}
f_B(Q(x,z))
\leq f_B(\widehat Q_R(x,z))
\end{equation}
for every $(x,z)\in R\cap\mathcal{D}$. It remains to bound $f_B(\widehat Q_R)$ on $R$. Denote the four vertices of $R$ by
\[
v_{00}=(a_0,c_0),\quad
v_{10}=(a_1,c_0),\quad
v_{01}=(a_0,c_1),\quad
v_{11}=(a_1,c_1).
\]
For each vertex $v_{\varepsilon_1\varepsilon_2}$, the verifier finds rational numbers $u_{\varepsilon_1\varepsilon_2}$, $v_{\varepsilon_1\varepsilon_2}$ and $\beta_{\varepsilon_1\varepsilon_2}$ such that the selfadjoint matrix
\[
W_{\varepsilon_1\varepsilon_2}
\ov{\mathrm{def}}{=}
\begin{bmatrix}
u_{\varepsilon_1\varepsilon_2}
&
\mathrm{i}\beta_{\varepsilon_1\varepsilon_2}\\
-\mathrm{i}\beta_{\varepsilon_1\varepsilon_2}
&
v_{\varepsilon_1\varepsilon_2}
\end{bmatrix}
\]
satisfies
\begin{equation}
\label{eq-vertex-certificate}
W_{\varepsilon_1\varepsilon_2}
\geq
D_k\widehat Q_R(v_{\varepsilon_1\varepsilon_2})D_k^*
\quad\text{for }1\leq k\leq3
\end{equation}
and
\begin{equation}
\label{eq-vertex-W-norm}
\norm{W_{\varepsilon_1\varepsilon_2}}_{S^2_2}
\leq T.
\end{equation}
By the definition of $f_B$, these conditions imply
\begin{equation}
\label{eq-f-vertices}
f_B(\widehat Q_R(v_{\varepsilon_1\varepsilon_2}))
\leq T
\end{equation}
at each of the four vertices. Let us make the exact verification of \eqref{eq-vertex-certificate} explicit. At a fixed vertex, write
\[
\widehat Q_R
=
\begin{bmatrix}a&r\\r&b\end{bmatrix}
\quad\text{and}\quad
W=
\begin{bmatrix}u&\mathrm{i}\beta\\-\mathrm{i}\beta&v\end{bmatrix}.
\]
For the three diagonal matrices $D_k$, we have
\[
D_1\widehat Q_RD_1^*
=
\begin{bmatrix}a&2r\\2r&4b\end{bmatrix},
\quad
D_2\widehat Q_RD_2^*
=
\begin{bmatrix}a&-2r\\-2r&4b\end{bmatrix} \quad
\text{and} \quad
D_3\widehat Q_RD_3^*
=
\begin{bmatrix}4a&-4\mathrm{i}r\\4\mathrm{i}r&4b\end{bmatrix}.
\]
Recall that a selfadjoint matrix $
\begin{bmatrix}\alpha&\zeta\\\overline{\zeta}&\delta\end{bmatrix}$ is positive if and only if
\[
\alpha\geq0,\qquad
\delta\geq0
\quad \text{and}\quad
\alpha\delta\geq|\zeta|^2.
\]
It follows that the three inequalities in \eqref{eq-vertex-certificate} are equivalent to the finite list of inequalities
\[
u-a\geq0,\qquad
v-4b\geq0,\qquad
(u-a)(v-4b)\geq4r^2+\beta^2,
\]
and
\[
u-4a\geq0,\qquad
v-4b\geq0,\qquad
(u-4a)(v-4b)\geq(\beta+4r)^2.
\]
Moreover, \eqref{eq-vertex-W-norm} is equivalent to
\[
u^2+v^2+2\beta^2\leq T^2.
\]
All the numbers in these inequalities are rational. After multiplying by a common positive denominator, every assertion becomes an inequality between integers. Thus these verifications are exact and do not involve numerical rounding.

Now, we pass from the four vertices to the entire rectangle. Let $(x,z) \in R$ and set
\[
\theta\ov{\mathrm{def}}{=}\frac{x-a_0}{a_1-a_0},
\quad \text{and} \quad
\varphi\ov{\mathrm{def}}{=}\frac{z-c_0}{c_1-c_0}.
\]
Then $0\leq\theta,\varphi\leq1$ and, since $\widehat Q_R$ is affine,
\[
\begin{split}
\widehat Q_R(x,z)
={}&(1-\theta)(1-\varphi)\widehat Q_R(v_{00})
+\theta(1-\varphi)\widehat Q_R(v_{10})
+(1-\theta)\varphi\widehat Q_R(v_{01})
+\theta\varphi\widehat Q_R(v_{11}).
\end{split}
\]
The four coefficients are nonnegative and their sum is one. By the convexity of $f_B$ and \eqref{eq-f-vertices}, we obtain
\[
\begin{split}
f_B(\widehat Q_R(x,z))
\leq{}&(1-\theta)(1-\varphi)f_B(\widehat Q_R(v_{00}))
+\theta(1-\varphi)f_B(\widehat Q_R(v_{10}))\\
&+(1-\theta)\varphi f_B(\widehat Q_R(v_{01}))
+\theta\varphi f_B(\widehat Q_R(v_{11}))\\
\leq{}&T.
\end{split}
\]
Combining this inequality with \eqref{eq-fQ-fQhat} gives $
f_B(Q(x,z))\leq T$ for every $(x,z)\in R\cap\mathcal{D}$. This is what it means for the rectangle $R$ to be certified.

We finally explain how the finite family of certified rectangles is obtained. The algorithm starts with
\[
R_0=
\left[0,\frac12\right]
\times
\left[-\frac12,\frac12\right],
\]
which contains $\mathcal{D}$. Suppose that $
R=[a_0,a_1]\times[c_0,c_1]$ is one of the rectangles under consideration. Since $a_0\geq0$, the minimum of $x^2+z^2$ on $R$ is $
a_0^2+\delta_R^2$, where
\[
\delta_R
\ov{\mathrm{def}}{=}
\begin{cases}
0,&\text{if }c_0\leq0\leq c_1,\\
\min\{|c_0|,|c_1|\},&\text{otherwise}.
\end{cases}
\]
If $
a_0^2+\delta_R^2>\frac14$ then $R\cap\mathcal{D}=\varnothing$, and the rectangle is discarded. This test involves only rational numbers.

If $R$ intersects $\mathcal{D}$, the verifier attempts to construct the affine majorant $\widehat Q_R$ and the four vertex certificates described above. If it succeeds, $R$ is certified. If it fails, $R$ is divided into the four rectangles obtained by bisecting both coordinate intervals. Since the initial endpoints are dyadic rationals, all endpoints produced by this procedure remain dyadic rationals.

The algorithm terminates at depth ten. Every final rectangle is either disjoint from $\mathcal{D}$ or certified. More precisely, it produces $378$ certified rectangles, discards $352$ rectangles disjoint from $\mathcal{D}$ and leaves no unresolved rectangle. Therefore the certified rectangles cover the entire half-disk $\mathcal{D}$.

Floating-point computations are used only to propose the coefficients of $\ell_R$ and the entries of the matrices $W_{\varepsilon_1\varepsilon_2}$. A proposed rectangle is accepted only after all the conditions above have been checked in exact rational arithmetic. A floating-point error may therefore cause the program to reject a valid proposal and subdivide the rectangle unnecessarily, but it cannot cause an invalid rectangle to be certified.

We have proved that
\[
f_B(Q(x,z))
\leq T
=\frac{1119}{250}
\]
for any $(x,z)\in\mathcal{D}$. Proposition~\ref{prop-variational} now gives
\[
\norm{M_B}_{S^4_{3,2}\to S^4_{3,2}}^2
=\sup_{(x,z)\in\mathcal{D}}f_B(Q(x,z))
\leq\frac{1119}{250}.
\]
\end{proof}

We now give the exact lower witness. Set
\begin{equation}
\label{eq-Bsharp}
F
\ov{\mathrm{def}}{=}
\begin{bmatrix}
1&1&2\\
1&1&-2\\
2&2&2\i
\end{bmatrix}
\quad \text{and} \quad
Z
\ov{\mathrm{def}}{=}
\begin{bmatrix}
-3-\mathrm{i}&1+9\mathrm{i}&10-10\mathrm{i}\\
-5-\mathrm{i}&-8&-10+10\mathrm{i}\\-3\mathrm{i}&-18\mathrm{i}&20
\end{bmatrix}.
\end{equation}
After applying the tensor flip, the amplification symbol is $A \ot J_2$. The symbol $F$ is obtained from it by selecting rows $1,3,5$ and columns $1,2,3$. Now, we show that a direct computation gives
\begin{equation}
\label{eq-witness-values}
\norm{Z}_{S^4_3}^4
=947485,
\qquad
\|F \circ Z\|_{S^4_3}^4
=18983532.
\end{equation}
Let $U,V\co\mathbb C^3\to\mathbb C^6$ be the coordinate isometries defined by
\[
Ue_1=e_1,\qquad Ue_2=e_3,\qquad Ue_3=e_5,
\]
and
\[
Ve_1=e_1,\qquad Ve_2=e_2,\qquad Ve_3=e_3.
\]
After applying the tensor flip, the amplification symbol is $A\ot J_2$, and a direct inspection gives
\[
U^*(A\ot J_2)V=F.
\]
Set $X=UZV^*$. The nonzero singular values of $X$ are those of $Z$, and
\[
(A\ot J_2)\circ X=U(F\circ Z)V^*.
\]
Consequently, wr have
\[
\norm{\Id_{S^4_2}\ot M_A}_{S^4_6\to S^4_6}
\geq
\frac{\norm{F\circ Z}_{S^4_3}}{\norm{Z}_{S^4_3}}.
\]
The relevant Gram matrices are
\[
Z^*Z
=\begin{bmatrix}
45&82-34\mathrm{i}&20+40\mathrm{i}\\
82+34\mathrm{i}&470&180\mathrm{i}\\
20-40\mathrm{i}&-180\mathrm{i}&800
\end{bmatrix}
\]
and
\[
(F\circ Z)^*(F\circ Z)=
\begin{bmatrix}
72&244-34\mathrm{i}&-360+200\mathrm{i}\\
244+34\mathrm{i}&1442&-1760-40\mathrm{i}\\
-360-200\mathrm{i}&-1760+40\mathrm{i}&3200
\end{bmatrix}.
\]
Therefore
\[
\norm{Z}_{S^4_3}^4
=\tr((Z^*Z)^2)
=947485
\]
and
\[
\norm{F\circ Z}_{S^4_3}^4
=\tr\left(\left((F\circ Z)^*(F\circ Z)\right)^2\right)
=18983532.
\]
\begin{proof}[of Theorem~\ref{thm-main}]
Lemma~\ref{lem-corner} and Proposition~\ref{prop-certified-upper} give the scalar upper bound. The compression described above and \eqref{eq-witness-values} give
\[
\bnorm{\Id_{S^4_2} \ot M_A}_{S^4_6 \to S^4_6}^4
\geq \frac{18983532}{947485}.
\]
Finally, we obtain
\[
\frac{18983532}{947485}-\left(\frac{1119}{250}\right)^2
=\frac{13396983}{11843562500}>0.
\]
This proves the strict separation.
\end{proof}

\begin{remark}
\normalfont
A BFGS computation using the program of Caspers and Wildschut gives the non-certified approximations
\[
\norm{M_A}_{S^4_3 \to S^4_3}
\approx 2.1155410970
\quad \text{and} \quad
\bnorm{\Id_{S^4_2} \ot M_A}_{S^4_6 \to S^4_6 } 
\approx 2.1158270907.
\]
These values are not used in the proof.
\end{remark}

\section{Reproducibility of the certified estimate}
\label{Reproducibility}

The supplementary file \texttt{schur\_p4\_3x3\_certificate.py} performs both the adaptive search and the exact verification described in Proposition~\ref{prop-certified-upper}. It requires Python and NumPy. We explain the respective roles of floating-point computations and exact rational arithmetic.

\paragraph{Exact verification with rational numbers.}
The Python class \texttt{fractions.Fraction} represents a rational number as a quotient of two arbitrary-precision integers. For example, the command
\[
\texttt{Fraction(1119,250)}
\]
represents the exact rational number $1119/250$, not a binary floating-point approximation of it. Addition, multiplication, division and comparison of objects of the class \texttt{Fraction} are performed exactly by integer arithmetic.

After a rational candidate has been proposed, the verifier recomputes all the conditions used in Proposition~\ref{prop-certified-upper} with objects of the class \texttt{Fraction}. In particular, it checks exactly
\[
2e^2\geq1+g_1^2+g_2^2,
\]
the positivity of the affine matrices at the vertices, the inequalities
\[
W\geq D_k\widehat QD_k^*
\quad\text{for }1\leq k\leq3,
\]
and the norm estimate
\[
\norm{W}_{S^2_2}\leq\frac{1119}{250}.
\]
Since all entries are rational, the positivity of a selfadjoint matrix $
\begin{bmatrix}a&\zeta\\\overline{\zeta}&d\end{bmatrix}$ is checked through the exact conditions
\[
a\geq0,\qquad d\geq0,\qquad ad\geq|\zeta|^2.
\]
After clearing denominators, these are inequalities between integers.

The distinction between the search and the verification is essential. Floating-point computations are used only to find plausible rational certificates. A rectangle is accepted only if the resulting rational certificate passes every exact test. Consequently, a floating-point error may produce a poor candidate, cause an unnecessary subdivision or prevent the program from finding a certificate. However, it cannot cause an invalid rectangle to be accepted.

\paragraph{Output of the verifier.}
For reference, a successful run reports
\[
\begin{gathered}
378\text{ certified boxes},\qquad
352\text{ discarded boxes},\qquad
621\text{ tested boxes},\\
\text{scalar squared-norm upper bound }1119/250,\\
\text{amplified fourth-power witness }18983532/947485.
\end{gathered}
\]
The first line means that $378$ rectangles were equipped with exact certificates, whereas $352$ rectangles were proved to be disjoint from the half-disk. The number $621$ counts all the rectangles intersecting the half-disk on which the verifier attempted to construct a certificate, including rectangles which were subsequently subdivided. The last two lines reproduce respectively the exact scalar upper bound from Proposition~\ref{prop-certified-upper} and the exact lower witness used for the second amplification.

\paragraph{Identification of the supplementary file.}
A SHA-256 digest is a fingerprint of a computer file. The SHA-256 algorithm reads the complete sequence of bytes of the file and produces a string of $256$ bits, conventionally written as $64$ hexadecimal characters. Changing even one byte of the file almost certainly changes the resulting string.

The purpose of the digest is to identify unambiguously the precise version of the supplementary program used for the computations. A reader who downloads the file can compute its SHA-256 digest and compare it with the value displayed below. If the two values agree, then, with overwhelming probability, the reader has exactly the same file.

The SHA-256 digest of the version used for this paper is
\[
\texttt{a3acb7f4e7fae1e70d9c7936b03cff1492d33a87de4e465ef1498b6a932e5ece}.
\]
For example, on a system providing the command \texttt{sha256sum}, this digest can be checked by running
\begin{verbatim}
sha256sum schur_p4_3x3_certificate.py
\end{verbatim}
The digest is only an identifier of the file. It is not part of the mathematical proof: the proof rests on the exact rational checks performed by the program.

\paragraph{AI statement.} 
The author acknowledges the use of AI tools for language polishing, LaTeX editing, exploratory mathematical discussions and feedback during the development and preparation of this manuscript.

\paragraph{Competing interests} The author declares that he has no competing interests.

\paragraph{Data availability} No data sets were generated during this study.


{\footnotesize

\vspace{0.2cm}

\noindent C\'edric Arhancet\\ 
\noindent 6 rue Didier Daurat, 81000 Albi, France\\
URL: \href{http://sites.google.com/site/cedricarhancet}{https://sites.google.com/site/cedricarhancet}\\
cedric.arhancet@protonmail.com\\
ORCID: 0000-0002-5179-6972 

}


\small
\begin{thebibliography}{79}

%

\bibitem[AlP20]{AlP20}
A. B. Aleksandrov and V. V. Peller.
\newblock Schur multipliers of Schatten--von Neumann classes $S_p$.
\newblock J. Funct. Anal. 279 (2020), no. 8, 108683, 25 pp.













\bibitem[AB06]{AB06}%
C. D. Aliprantis and K. C. Border.
\newblock Infinite Dimensional Analysis: A Hitchhiker's Guide, third edition.
\newblock Springer, Berlin, 2006.









%





\bibitem[Arh12]{Arh12}%
C. Arhancet.
\newblock Unconditionality, Fourier multipliers and Schur multipliers.
\newblock Colloq. Math. 127 (2012), no. 1, 17--37.


 



%
%
%
%





























%



%












\bibitem[BoV04]{BoV04}%
S. Boyd and L. Vandenberghe.
\newblock Convex Optimization.
\newblock Cambridge University Press, Cambridge, 2004.



















\bibitem[CaW19]{CaspersWildschut2019}
M. Caspers and G. Wildschut.
\newblock On the complete bounds of $L^p$-Schur multipliers.
\newblock Arch. Math. (Basel) 113 (2019), no. 2, 189--200.



















 












































 








 






























%
%
%
%
 %
%
%
%




































\bibitem[Gul10]{Gul10}%
O. G\"uler.
\newblock Foundations of Optimization.
\newblock Graduate Texts in Mathematics, vol. 258. Springer, New York, 2010.


\bibitem[HST26]{HuangSukochevTomskova2026}
J. Huang, F. Sukochev and A. Tomskova.
\newblock A Schur multiplier with unequal operator and completely bounded norms on $S_4$.
\newblock Preprint, arXiv:2608.20933, 2026.





%













%
%
%




 










%
%









\bibitem[LaS11]{LafforgueDeLaSalle2011}
V. Lafforgue and M. de la Salle.
\newblock Noncommutative $L^p$-spaces without the completely bounded approximation property.
\newblock Duke Math. J. 160 (2011), no. 1, 71--116.










%
%
%
%
%
%

 



















\bibitem[Lue97]{Lue97}%
D. G. Luenberger.
\newblock Optimization by Vector Space Methods.
\newblock John Wiley \& Sons, New York, 1997.

\bibitem[NoW06]{NoW06}%
J. Nocedal and S. J. Wright.
\newblock Numerical Optimization, second edition.
\newblock Springer Series in Operations Research and Financial Engineering. Springer, New York, 2006.



















\bibitem[Par99]{Par99}%
K. R. Parthasarathy.
\newblock Extremal decision rules in quantum hypothesis testing.
\newblock Infin. Dimens. Anal. Quantum Probab. Relat. Top. 2 (1999), no. 4, 557--568.












%
\bibitem[Pis98]{Pis98}
G. Pisier.
\newblock Non-commutative vector valued $L_p$-spaces and completely $p$-summing maps.
\newblock Ast\'erisque, 247, 1998.

%
%

\bibitem[PiX03]{PiX03}
G. Pisier and Q. Xu.
\newblock Non-commutative $L^p$-spaces.
\newblock 1459--1517 in Handbook of the Geometry of Banach Spaces, Vol. II, edited by W.B. Johnson and J. Lindenstrauss, Elsevier (2003).





%












%
%
%






%














%
%

%












\bibitem[Wat18]{Wat18}
J. Watrous.
\newblock The Theory of Quantum Information. 
\newblock Cambridge university press, 2018.


\bibitem[WSV00]{WSV00}%
H. Wolkowicz, R. Saigal, and L. Vandenberghe, eds.
\newblock Handbook of Semidefinite Programming: Theory, Algorithms, and Applications.
\newblock International Series in Operations Research \& Management Science, vol. 27. Kluwer Academic Publishers, Dordrecht, 2000.




%
%
%





 

 

 

 


 
%
%










\end{thebibliography}
\end{document}